\documentclass[11pt]{amsart}
\usepackage{amsmath,amscd,amssymb,amsfonts,amsthm}

\usepackage[cmtip,matrix,arrow]{xy}

\newtheorem{theorem}{Theorem}[section]
\newtheorem{corollary}[theorem]{Corollary}
\newtheorem{proposition}[theorem]{Proposition}

\newtheorem{lemma}[theorem]{Lemma}
\theoremstyle{definition}
\newtheorem{definition}[theorem]{Definition}
\theoremstyle{remark}

\newtheorem{remark}[theorem]{Remark}
\newtheorem{remarks}[theorem]{Remarks}
\newtheorem{example}[theorem]{Example}

\newcommand{\Cour}[1]      {[\![#1]\!]}

\newcommand\G{\mathcal{G}}

\newcommand{\T}{\mathbb{T}}

\newcommand{\ca}{\mathcal}

\newcommand{\E}{\ca{E}}

\newcommand{\R}{\mathbb{R}}
\newcommand{\C}{\mathbb{C}}

\renewcommand{\P}{\mathsf{P}}

\renewcommand{\a}{\mathsf{a}}

\newcommand{\on}{\operatorname}
\newcommand{\Aut}{ \on{Aut} }

\renewcommand{\subset}{\subseteq}

\renewcommand{\ker}{ \on{ker}}

\newcommand\qu{/\kern-.7ex/} 

\newcommand{\hra}{\hookrightarrow}

\renewcommand{\d}{{\mbox{d}}}
\newcommand{\ol}{\overline}

\newcommand{\dd}{\mf{d}}

\newcommand\sig{\sigma}
\newcommand\eps{\epsilon}
\newcommand\Om{\Omega}
\newcommand\om{\omega}

\newcommand{\f}{\frac}

\newcommand{\p}{\partial}
\renewcommand{\l}{\langle}
\renewcommand{\r}{\rangle}
\newcommand\hh{{\f{1}{2}}}

\newcommand{\eeq}{\end{eqnarray*}}
\newcommand{\beq}{\begin{eqnarray*}}

\newcommand{\End}{\on{End}}
\newcommand{\pr}{\on{pr}}

\newcommand{\wt}{\widetilde}
\newcommand{\mf}{\mathfrak}

\newcommand{\rra}{\rightrightarrows}

\newcommand{\aut}{\mf{aut}}
\newcommand{\gau}{\mf{gau}}
\newcommand{\da}{\dasharrow}

\begin{document}
\sloppy
\title{Splitting theorems for Poisson and related structures}
\author{Henrique Bursztyn}
\author{Hudson Lima}
\address{Instituto de Matem\'atica Pura e Aplicada,
Estrada Dona Castorina 110, Rio de Janeiro, 22460-320, Brasil }
\email{henrique@impa.br}
\email{hudsonnl@impa.br}
\author{Eckhard Meinrenken}
\address{Department of Mathematics,
University of Toronto,
40 St. George Street,
Toronto, Ontario, M5S 2E4}
\email{mein@math.toronto.edu}

\begin{abstract}
 According to the Weinstein splitting theorem, any Poisson manifold is locally, near any given point, a product of a symplectic manifold with another Poisson manifold whose Poisson structure vanishes at the point. Similar splitting results are known e.g. for Lie algebroids, Dirac structures and generalized complex structures. In this paper, we develop a novel approach towards these results that leads to various generalizations, including their equivariant versions as well as their formulations in new contexts.
\end{abstract}
\maketitle
\section{Introduction}\label{sec:intro}
Poisson manifolds and related geometric structures, such as Lie algebroids, Dirac structures and generalized complex structures, display an intricate local theory. The \emph{splitting theorems} to be discussed in this paper refer to a series of results that provide fundamental local information about these types of geometry.

In each of these contexts, the geometric structure on the given manifold $M$ determines a  generalized foliation of $M$, in the sense of Stefan and Sussmann. While the leaves of such a \emph{foliated manifold} need not be of constant dimension, the Stefan-Sussmann theory shows that they are arranged rather nicely: For every $m\in M$ there is an open neighborhood isomorphic to a product of foliated manifolds  $S\times N$, where $S$ has the trivial foliation (with $S$ itself as its only leaf)
while $N$ contains the point $m$ as a zero-dimensional leaf. The splitting theorems say that, in each case,
one can take this splitting $S\times N$ to be compatible with the given geometric structure. The following are some instances of such results:
\begin{enumerate}
\item
Weinstein's splitting theorem \cite{wei:loc} for \emph{Poisson manifolds} $(M,\pi)$, which  asserts the existence of a neighborhood of $m$ that is Poisson diffeomorphic to a product $(S,\pi_S)\times (N,\pi_N)$, where $\pi_S$ is non-degenerate while $\pi_N$ vanishes at $m$;

\item the splitting theorem for \emph{Dirac manifolds} \cite{cou:di}, obtained by Blohmann \cite{blo:rem} (see also Dufour-Wade \cite{duf:lo} for related results);

\item the splitting theorem for Lie algebroids $E\to M$, due to Dufour \cite{duf:nor}, Fernandes \cite{fer:lie},  and Weinstein  \cite{wei:alm},
which gives an isomorphism near $m$ with a product of Lie algebroids $TS \times F$,
where the anchor of the Lie algebroid $F\to N$ vanishes at $m$;

\item
the splitting theorem for generalized complex manifolds \cite{gua:ge}, due to
Abouzaid-Boyarchenko \cite{abo:lo}, which shows that up to a $B$-field transform, any generalized complex manifold is locally a product $S\times N$ of generalized complex manifolds, where $S$ is `of symplectic type' and $N$ is `of complex type' at $m$.
\end{enumerate}

In this article, we develop a novel approach towards splitting theorems, which allows us to generalize them in various directions and to new contexts. Rather than taking $N$ to be `small', we will allow transverse submanifolds $N\hookrightarrow M$ that may be quite large. Transversality implies that the normal bundle $\nu_N$ inherits  a `linear approximation' of the given geometry. Our local models will give tubular neighborhood embeddings, identifying the geometric structures  over the normal bundle $\nu_N$ and over an open neighborhood of the transversal $N\subset M$. (This is not to be confused with linearization problems around leaves.) In the Poisson case, we recover
the normal form theorem of Frejlich-M\u{a}rcu\c{t} \cite{fre:nor}.

Our main technical tool is a linearization lemma for vector fields $X$ vanishing along submanifolds $N\subset M$, with linear approximation given by the Euler vector field on the normal bundle $\nu_N$. Any such `Euler-like' vector field determines a tubular neighborhood embedding, and the strategy of the proof is to make the vector field,
and hence the tubular neighborhood, compatible with the given geometric data.
A key feature of our approach is that constructions are quite explicit, in the sense that  normal forms are fully determined by some specific choices, with a natural dependence on them. As a result, they have good functorial properties, so that one obtains the $G$-equivariant versions of the normal form theorems without extra effort. We remark that it is unclear how to obtain equivariant splitting theorems from the traditional (induction-based) proofs. Indeed, for Poisson manifolds $(M,\pi)$, a $G$-equivariant Weinstein splitting theorem was only recently proved by Frejlich-M\u{a}rcu\c{t} in \cite{fre:nor}, following partial results in Miranda-Zung \cite{mir:no}. The argument in \cite{fre:nor} towards Weinstein splittings, and more generally normal forms along cosymplectic transversals  $N\subset M$, uses `Poisson sprays' and the approach of Crainic-M\u{a}rcu\c{t} \cite{cra:exi} to symplectic realizations. In contrast, our normal form for  Poisson case is entirely determined by the
choice of a 1-form $\alpha\in \Om^1(M)$ whose image under the map $\pi^\sharp\colon T^*M\to TM$ is an Euler-like vector field along $N$.

The structure of this paper is as follows. In Section \ref{sec:euler} we discuss the linearization of Euler-like vector fields and the resulting tubular neighborhood embeddings.
In Section \ref{sec:anchored} we apply this to anchored vector bundles satisfying an involutivity condition. We obtain a  normal form theorem along transversals, which may be regarded as a version of the Stefan-Sussmann theorem. This is followed by similar transversal normal form theorems for transversals of Lie algebroids in Section \ref{sec:lialg}, and Dirac structures in Section \ref{sec:dirac}, which are new in these
contexts.

From our result for Dirac structures, we derive as direct consequences the transversality
results for Poisson structures in Section \ref{sec:poisson}, and generalized
complex structures in Section \ref{sec:gcs}. Similar results for generalized
complex structures have independently been obtained in recent work of Bailey-Cavalcanti-Duran \cite{bai:blo}, using a different approach.  Our method also leads to new results on transverse normal forms for
Courant algebroids, but this case is less straightforward and  will be
treated separately.

In future work, we plan to generalize some of these techniques to infinite-dimensional settings. Indeed, one of our inspirations was the proof of Frobenius' theorem for Banach manifolds, in the books \cite{ab:ma} and \cite{lan:dif}, and the realization that the geometry behind these proofs involves the flow of an Euler-like vector field. We expect that similar techniques can be used to prove versions of the Stefan-Sussmann theorem and other splitting theorems for infinite-dimensional manifolds.

\vskip.1in
{\bf Acknowledgements.} We would like to thank David Li-Bland, Pedro Frejlich, Ioan M\u{a}rcu\c{t}  and Shlomo Sternberg, for helpful comments on various aspects of our work. We also thank the anonymous referee for his/her comments and suggestions. H.B. and H.L. thank Faperj and CNPq for financial support; E.M. was supported by an NSERC Discovery Grant.

\newpage
\section{Euler-like vector fields and tubular neighborhoods}\label{sec:euler}
\subsection{Notation and conventions}\label{subsec:notation}
For a manifold $M$, we denote by $\on{Diff}(M)$ the group of diffeomorphisms, and by $\mf{X}(M)$ the Lie algebra of vector fields. For a complete vector field $X\in \mf{X}(M)$, we define
its flow to be the 1-parameter group of diffeomorphisms $\Phi_s\in \on{Diff}(M)$ such that
$X(f)=\f{d}{d s}\big |_{s=0} (\Phi_s^{-1})^*f$ for $f\in C^\infty(M)$. The description also holds
for the local flows of incomplete vector fields, taking care of the domains. The flow of a time-dependent vector field $X_s$ is defined in terms of the action on functions by $\f{d}{d s} \Phi_s^*=-\Phi_s^*\circ X_s$ with $\Phi_0=\on{id}$.

Given a vector bundle  $\pr \colon E\to M$, we denote by
$\on{Aut}(E)\subset \on{Diff}(E)$ its group of  automorphisms, and by $\mf{aut}(E)\subset \mf{X}(E)$ the Lie algebra of infinitesimal automorphisms. Any $\wt{\Phi}\in \on{Aut}(E)$
restricts to a diffeomorphism $\Phi\in \on{Diff}(M)$ with $\pr\circ\, \wt{\Phi}=\Phi\circ \pr$
; the kernel of this restriction map  is denoted $\on{Gau}(E)$. Likewise, any $\wt{X}\in\mf{aut}(E)$ restricts to a vector field $X\in\mf{X}(M)$ with
\begin{equation}\label{eq:xlift}
\wt{X}\sim_{\pr} X;
\end{equation}
the kernel of this restriction map is denoted by $\mf{gau}(E)$. According to \cite{gra:hig}, the elements $\wt{\Phi}\in \on{Aut}(E)$ are characterized as those  diffeomorphisms of the total space of $E$ that commute with the action of the group $\R_{>0}$ by scalar multiplication; it is automatic that such a diffeomorphism preserves fibers and is fiberwise linear. Similarly $\mf{aut}(E)$ consists of vector fields on $E$ that are invariant under the action of $\R_{>0}$. Any $\wt{\Phi}\in \Aut(E)$ determines an invertible linear operator $A\colon \Gamma(E)\to \Gamma(E)$, taking $\sigma\colon M\to E$ to  $\wt{\Phi}\circ \sigma\circ \Phi^{-1}$. This has the property
\begin{equation}\label{eq:autprop}
A(f\sigma)=((\Phi^{-1})^*f)\ A(\sigma);
\end{equation}
conversely, any invertible linear operator with this property
corresponds to a unique $\wt{\Phi}\in\Aut(E)$ lifting $\Phi$. Infinitesimally, for  $\wt{X}\in\aut(E)$ one obtains a  `Lie derivative' of sections $D\colon \Gamma(E)\to \Gamma(E)$, with the derivation property
\begin{equation}\label{eq:derprop}
D(f\sigma)=f\,D\sigma+X(f)\sigma
\end{equation}
for all $f\in C^\infty(M)$. Conversely, any linear operator $D$ with this property
corresponds to a unique $\wt{X}\in\mf{aut}(E)$ lifting $X$. If $\wt{X}\in\mf{gau}(E)$ then
the corresponding $D$ is a section of the bundle $E^*\otimes E$ of endomorphisms of $E$.

\begin{example}\label{ex:euler}
Let  $\kappa_t\colon E\to E$ denote scalar multiplication by $t\in \mathbb{R}$, and let
$\E\in\gau(E)$ be the \emph{Euler vector field}.
In local bundle coordinates on $E$, with $x^i$ the coordinates in the fiber direction and $y^j$ those in the base direction,
\[ \E=\sum_i x^i\f{\p}{\p x^i}.\]
The flow of $\E$ is $s\mapsto \kappa_{\exp(-s)}$; hence the endomorphism of $E$ corresponding to $\E$ is $D=\kappa_{-1}$.
\end{example}


\subsection{Normal bundles and linear approximation}


Given a manifold $M$ and a submanifold $N\subset M$, let $\nu(M,N)=TM|_N/TN$
be the normal bundle. We write $\nu_N=\nu(M,N)$ if the ambient manifold is clear. Throughout this paper, $p$ and $i$ will denote the following projection and inclusion:
\begin{equation}\label{eq:pi}
 \xymatrixcolsep{10pt}
\xymatrixrowsep{18pt}
\xymatrix{ \nu(M,N) \ar[d]_p & \\ N\ar[r]_i & M. }
\end{equation}

Given a smooth map of pairs $\varphi\colon (M',N')\to (M,N)$ (that is, $\varphi\colon M'\to M$ is a smooth map with $\varphi(N')\subset N$), one obtains a vector bundle
morphism
\begin{equation}\label{eq:mapofpairs}
\nu(\varphi)\colon \nu(M',N')\to \nu(M,N)
\end{equation}
over $\varphi|_{N'}\colon N'\to N$,
with the obvious functorial property under composition of such maps. If $\varphi$ is transverse to $N$, and $N'=\varphi^{-1}(N)$, then  $\nu(\varphi)$ is a fiberwise isomorphism.

The normal bundle functor is compatible with the tangent functor:  There is a canonical isomorphism
\begin{equation}\label{eq:tnu} \nu(TM,TN)\xrightarrow{\cong} T\nu(M,N) \end{equation}
identifying the structures as vector bundles over $\nu(M,N)$ and also as vector bundles over $TN$.  In other words, \eqref{eq:tnu} is an isomorphism between the following two \emph{double vector bundles}:
\[ \xymatrixcolsep{30pt}
\xymatrixrowsep{18pt}
\xymatrix{{\nu(TM,TN)} \ar[r]^{}\ar[d]_{}    & TN \ar[d]\\
{\nu(M,N)}\ar[r]_{} & N,}
\ \ \  \ \ \ \ \ \ \
\xymatrix{
{T\nu(M,N)} \ar[r]^{\;Tp}\ar[d]_{}    & TN \ar[d]\\
{\nu(M,N)}\ar[r]_{} & N.}
\]
For any map of pairs $\varphi\colon (M',N')\to (M,N)$, the following diagram commutes:
\[ \xymatrixcolsep{30pt}
\xymatrixrowsep{18pt}
\xymatrix{{\nu(TM',TN')} \ar[r]^{\cong}\ar[d]_{\nu(T\varphi)}    & T\nu(M',N') \ar[d]^{T(\nu(\varphi))}\\
{ \nu(TM,TN)}\ar[r]_{\cong} & T\nu(M,N).}
\]
See Appendix \ref{app:tnu}  for a  detailed discussion.

Suppose that $E\to M$ is a vector bundle, and
$\sigma\in \Gamma(E)$ is a smooth section with
$\sigma|_N=0$. Then $\sigma\colon (M,N)\to (E,M)$ induces a vector bundle map
$\nu(\sigma)\colon \nu(M,N)\to \nu(E,M)$. Making use of the natural identification $\nu(E,M)\cong E$,
we obtain a vector bundle map
\begin{equation}
\label{eq:normal}\d^N\sigma\colon \nu(M,N)\to E|_N,
\end{equation}
referred to as the \emph{normal derivative} (or \emph{intrinsic derivative} \cite{gu:gea})
of $\sigma$, since it codifies the
derivative of $\sigma$ in directions normal to $N$. Using a partition of unity, it is clear that every bundle map $\nu(M,N)\to E|_N$ arises in this way, as the normal derivative $\d^N\sigma$ of some section.

For a diffeomorphism $\Phi$ of $M$ preserving $N$, the map $\Phi\colon (M,N)\to (M,N)$  defines the \emph{linear approximation} $\nu(\Phi)\in \on{Aut}(\nu_N)$. Infinitesimally, for a vector field $X\in\mf{X}(M)$ tangent to $N$, the map $X\colon (M,N)\to (TM,TN)$ induces
$\nu(X)\colon \nu(M,N)\to \nu(TM,TN)$. Using the identification
\eqref{eq:tnu} this is a vector field on $\nu_N$, called the  \emph{linear approximation} of $X$:
\begin{equation}\label{eq:linapp}
\nu(X)\in \mf{aut}(\nu_N).
\end{equation}
\begin{remark}\label{rem:linearapprox}
The linear approximation $\nu(X)$ can be viewed in alternative ways:
\begin{enumerate}
\item[(i)] The local flow of $\nu(X)$ is the linear approximation of the local flow of $X$.
\item[(ii)] The operator $D\colon \Gamma(\nu_N)\to \Gamma(\nu_N)$, corresponding
to  $\nu(X)\in\mf{aut}(\nu_N)$ as in Sec.~\ref{subsec:notation},
has the following description:  If $\tau\in \Gamma(\nu_N)$ is represented by a vector field $Y\in \mf{X}(M)$ (modulo a vector field tangent to $N$), then
$D(\tau)$ is represented by the Lie bracket $[X,Y]$.
\item[(iii)]  If  $X|_N=0$, then $\nu(X)\in \mf{gau}(\nu_N)$ is given by $ - \d^N X\colon \nu_N\to TM|_N$, followed by the projection $TM|_N\to \nu_N$.
    \end{enumerate}
\end{remark}

Recall that the \emph{tangent lift} $X_T\in \mf{X}(TM)$ of a vector field $X\in\mf{X}(M)$ is obtained by applying the tangent functor to
$X\colon M\to TM$ (more precisely, $X_T=J\circ TX$, where $J$ is the canonical involution on $TTM$, see Appendix \ref{app:tnu}). Equivalently, its local flow is the differential $T\Phi_s$ of the local flow $\Phi_s$ of $X$.  If $X$ is tangent to $N$, then the infinitesimal version of the identification $\nu(T\Phi_s)=T(\nu(\Phi_s))$
shows that
\begin{equation}\label{eq:xtanlift}\nu(X_T)=\nu(X)_T\end{equation}
as vector fields on $\nu(TM,TN)=T\nu(M,N)$.


\subsection{Tubular neighborhood embeddings}

Let $N\subset M$ be a submanifold, with normal bundle $\nu_N=\nu(M,N)$.
We will work with the following strong notion of tubular neighborhood embeddings.
\begin{definition} A \emph{tubular neighborhood embedding} for $N\subset M$ is an
embedding
\[ \psi\colon \nu_N\to M,\]
taking the zero section of $\nu_N$ to $N$,  and such that the map $\nu(\psi)$ induced by $\psi\colon (\nu_N,N)\to (M,N)$ is the identity map on $\nu_N$.
\end{definition}
Here we are making use of the canonical identification $\nu(\nu_N,N)=\nu_N$
given by the vector bundle structure. Note that some authors only require that $\psi|_N$ is the identity, rather than also the linear approximation $\nu(\psi)$.
A vector field $X$ tangent to $N$ is called \emph{linearizable} if there exists a tubular neighborhood embedding
$\psi$ such that $\nu(X)$ agrees with $\psi^*X$ on a neighborhood of $N$. We will need linearizability for the following special case.
\begin{lemma}\label{lem:linearization}
Suppose that $X|_N=0$, with linear approximation $\nu(X)=\E$ the Euler vector field  on $\nu_N$. Then $X$ is linearizable.
\end{lemma}
\begin{proof}
By choosing an initial tubular neighborhood embedding $\nu_N\hra M$, we may assume that $M=\nu_N$, and  that the difference
\[
Z=\E-X\in\mf{X}(M)
\]
has linear approximation equal to zero.  The family of vector fields
\[ Z_t=\f{1}{t} \kappa_t^*Z,\ \ \ \ t>0,\]
extends smoothly to $t=0$. \footnote{In local bundle coordinates on $\nu_N$,
with $x^i$ the coordinates in the fiber direction and $y^j$ those in the base direction, we have
$Z=\sum_i  g^i(x,y)\f{\p}{\p x^i}+\sum_j h^j(x,y)\f{\p}{\p y^j}$,
where $x\mapsto g^i(x,y)$ vanishes to second order at $x=0$ (as a consequence of
$\nu(Z)=0$), and $x\mapsto h^j(x,y)$ vanishes to first order (since $Z|_N=0$).  Hence $Z_t=\f{1}{t^2}\sum_i  g^i(tx,y)\f{\p}{\p x^i}+\f{1}{t}\sum_j h^j(tx,y)\f{\p}{\p y^j}$ extends to $t=0$.}
Let $\varphi_t$ be the flow of the time-dependent vector field $Z_t$, with $\varphi_0=\on{id}$. Since $Z_t|_N=0$, the set of points $m$ such that the integral curve $\varphi_t(m)$ exists for time $0\le t\le 1$, is an open neighborhood of $N$ in $\nu_N$.
By the scaling property $\kappa_a^*Z_t=a Z_{at}$ for $0<a<1$, this neighborhood is invariant under $\kappa_t$ for $0\le t\le 1$. Using that $\kappa_t^*\E=\E$,  and
$t\f{d}{d t}\kappa_t^*Y=\kappa_t^*[\E,Y]$ for all vector fields $Y$, we obtain
\[ \begin{split}
\f{d}{d t}\varphi_t^*(\E-t\,Z_t)
&=\f{d}{d t}\varphi_t^*(\E-\kappa_t^* Z)\\
&=\varphi_t^*\Big(-[Z_t,\E-\kappa_t^*Z]-\f{1}{t} \kappa_t^*[\E,Z]\Big)\\
&=\varphi_t^*(-[Z_t,\E]-[\E,Z_t])=0.
\end{split}
\]
Hence $\varphi_t^*(\E-t Z_t)$ does not depend on $t$. Equality of the values at $t=1$ and $t=0$ gives $\varphi_1^*(X)=\E$.
Hence, any tubular neighborhood embedding that agrees with $\varphi_1$ near $N$ will give the desired linearization.
\end{proof}
\begin{remark}
The question of linearizability of vector fields is subtle, and has been extensively studied.
(See e.g. \cite{bas:lin} for a quick overview and recent results.) The classical result of Sternberg \cite{ste:loc,ste:str} gives $C^\infty$-linearizability of vector fields at critical points $m$, provided the endomorphism of $T_mM$ describing this linear approximation has \emph{non-resonant} eigenvalues. If the linear approximation is the Euler vector field, then this endomorphism is $-\on{id}$, and the non-resonance condition is satisfied.  Thus, for $N=\{m\}$, Lemma \ref{lem:linearization} reduces to a very special case of Sternberg's theorem.
\end{remark}

\begin{definition}
Let $N\subset M$ be a submanifold.
A vector field $X\in\mf{X}(M)$
is called \emph{Euler-like} (along  $N$) if it is complete, with $X|_N=0$, and its linear approximation is the Euler vector field: $\nu(X)=\E$.
\end{definition}
Given a tubular neighborhood embedding, the push-forward of  $\E$ under $\psi$ is an Euler-like vector field $X$ on the image $U=\psi(\nu_N)$. The tubular neighborhood
embedding itself can be recovered from $X$, by using its flow. In fact, we have the following precise result:
\begin{proposition}\label{prop:euler}
Suppose that  $X\in\mf{X}(M)$ is Euler-like along $N\subset M$. Then there exists a  unique tubular neighborhood embedding $\psi\colon \nu_N\to M$ such that
\[ \E\sim_\psi X.\]
Given an action of a Lie group $G$ on $M$, preserving the submanifold $N$ and the vector field $X$, then the tubular neighborhood embedding $\psi$ is $G$-equivariant.
\end{proposition}
\begin{proof}
The existence of such a tubular neighborhood embedding is clear from
Lemma \ref{lem:linearization}. To prove uniqueness, suppose that a tubular neighborhood
embedding $\psi$ satisfying $\E\sim_\psi X$ is given.

Let $\Psi_s$ be the flow of $\E$ and $\Phi_s$ the flow of $X$. Recall that $\Psi_s=\kappa_{\exp(-s)}$, where  $\kappa_t\colon \nu_N\to \nu_N$ denotes the scalar multiplication by $t\in\R$. Thus $\kappa_t=\Psi_{-\log(t)}$ for $t>0$;
accordingly we define $\lambda_t=\Phi_{-\log(t)}$. Since $\nu_N$ is invariant under
$\kappa_t$ for all $t>0$, its image $U=\psi(\nu_N)$ is invariant under $\lambda_t$
for all $t>0$. Furthermore, since $\lim_{t\to 0}\kappa_t$ is the retraction
$p$ from $\nu_N$ onto $N\subseteq \nu_N$, we have
\begin{equation}\label{eq:U}
 U=\{m\in M|\ \lim_{t\to 0}\lambda_t(m)\mbox{ exists and lies in } N\subset M\}.\end{equation}
Let $v\in \nu_N$, with image point $m=\psi(v)$, and put $x=\kappa_0(v)=\lambda_0(m)$.
Then $\kappa_t(v)$ is a smooth curve in $\nu_N$, defined for $t\ge 0$, and
$\lambda_t(m)$ its image under $\psi$. Thus
\[ T_x\psi\Big(\f{d}{d t}\Big|_{t=0} \kappa_t(v)\Big)=
\Big(\f{d}{d t}\Big|_{t=0} \lambda_t(m)\Big).\]
Since $\psi$ is a tubular neighborhood embedding, the map $\nu(\psi)$ on normal bundles is the identity map. Hence
\[ \Big(\f{d}{d t}\Big|_{t=0} \lambda_t(m)\Big)\!\!\mod T_xN
=\Big(\f{d}{d t}\Big|_{t=0} \kappa_t(v)\Big)\!\!\mod T_xN.\]
But the element on the right hand side is just $v\in \nu_N|_x$. Since $\psi(v)=m$,  this shows that
\begin{equation}\label{eq:psinv}
 \psi^{-1}(m)=\Big(\f{d}{d t}\Big|_{t=0} \lambda_t(m)\Big)\mod T_xN.
\end{equation}
Equations \eqref{eq:U} and \eqref{eq:psinv} express $U=\psi(\nu_N)$ and the inverse map
$\psi^{-1}\colon U\to \nu_N$, hence also $\psi$ itself,  in terms of the flow of the vector field $X$. This shows that $\psi$ is unique.

In the $G$-equivariant setting, it is immediate that \eqref{eq:psinv} is $G$-equivariant, hence so is $\psi$.
\end{proof}

\begin{remark}
A result similar to Proposition \ref{prop:euler} may be found in  \cite[Theorem 2.2]{gra:hig}. (One issue to be pointed out, however, is that the argument in \cite{gra:hig}, based on Shoshitaishvili's theorem on topological normal forms for vector fields,  does not apply to the $C^\infty$-case.)
\end{remark}

\begin{remark}\label{rem:bump}
Suppose that $X|_N=0$ with linearization $\nu(X)=\E$. Then we may multiply
 $X$ by a bump function supported on a neighborhood of $N$, and equal to $1$ on a smaller neighborhood, to arrange that $X$ is complete (and hence Euler-like). Indeed,
 by Lemma \ref{lem:linearization}  there is an open neighborhood of $N$ consisting of points $m$ with $\lim_{s\to \infty}\Phi_s(m)\in N$, and one only needs to take the bump function to be supported in such a neighborhood.
\end{remark}

\subsection{Functoriality}\label{subsec:functorialityTNE}
The following functorial property is immediate from the construction.
Suppose $\varphi\colon (M',N')\to (M,N)$ is a smooth map of pairs, defining a
vector bundle morphism $\nu(\varphi)$  as in \eqref{eq:mapofpairs}. Let $X,X'$ be
Euler-like vector fields along $N,N'$, respectively, with
\[ X'\sim_\varphi X.\]
Then the resulting tubular neighborhood embeddings give a commutative diagram:
\begin{equation}\label{diag:funct1} \xymatrixcolsep{50pt}\xymatrix{
{\nu(M',N')} \ar[r]^{\ \ \psi'}\ar[d]_{\nu(\varphi)}    & M'\ar[d]^\varphi\\
{\nu(M,N)}\ar[r]_{\ \ \psi} & M.
}
\end{equation}

\begin{example}\label{ex:taneuler}
If $X$ is Euler-like along $N$, then its tangent lift $X_T$ is Euler-like along $TN$.
Indeed,
\[ \nu(X_T)=\nu(X)_T=\E_T,\]
the tangent lift of the Euler vector field $\E\in\mf{X}(\nu_N)$.  Letting $\psi\colon \nu(M,N)\to M$ and $\psi_T\colon T\nu(M,N)\cong \nu(TM,TN)\to TM$
be the tubular neighborhood embeddings, we obtain a commutative diagram
\begin{equation}\label{diag:funct3} \xymatrixcolsep{50pt}\xymatrix{
{\nu(TM,TN)} \ar[r]^{\ \ \psi_T}\ar[d]_{}    & TM\ar[d]\\
{\nu(M,N)}\ar[r]_{\ \ \psi} & M.
}
\end{equation}
But, upon the identification $\nu(TM,TN)\cong T\nu(M,N)$, we see that $\E_T$  is just the Euler vector field on $\nu(TM,TN)$, because the tangent lift of scalar multiplication on $\nu(M,N)$ is scalar multiplication on $\nu(TM,TN)$.  It follows that
 $\psi_T$ is simply the differential $T\psi$.
\end{example}

\section{Anchored vector bundles}\label{sec:anchored}

As our first application of Euler-like vector fields, we will obtain a normal form theorem for integrable anchored vector bundles.  This result may be regarded as a version of the Stefan-Sussmann theorem for generalized distributions on manifolds.

\subsection{Basic definitions}
A {\em smooth generalized distribution} on $M$, in the sense of Stefan \cite{ste:int} and Sussmann \cite{sus:orb}, is a collection
$D=\bigcup_{m\in M} D_m$ of subspaces $D_m\subset T_mM$, with the following property:
There exists a  submodule $\mf{C}\subset \mf{X}(M)$ of the $C^\infty(M)$-module of vector fields, such that $D_m$ is the image of $\mf{C}$ under
evaluation $\mf{X}(M)\to T_mM$.  If $m\mapsto \dim(D_m)$ is constant, then $D$ is a vector subbundle of $TM$, referred to as a \emph{regular distribution}.

Given a vector bundle $E\to M$ equipped with an {\em anchor}, i.e., a bundle map $\a\colon E\to TM$ covering the identity map, the image $D=\a(E)$ is always a smooth generalized distribution with $\mf{C}=\a(\Gamma(E))$. By a result of \cite{dra:smo}, any smooth generalized distribution arises in this way, though in general there is no canonical choice for the  vector bundle and anchor. In many geometric situations, however, vector bundles and anchors are naturally present.

\begin{definition}
An \emph{anchored vector bundle} is a vector bundle $E\to M$  together
with a vector bundle morphism (called the \emph{anchor}) $\a\colon E\to TM$, with base map the identity map.  A {\em morphism} from an anchored vector bundle $F\to N$ to
an anchored vector bundle $E\to M$ is  a bundle map $\wt{\varphi}\colon F \to E$, with base map $\varphi\colon N \to M$, such that the following diagram commutes:
\begin{equation}\label{eq:morphismAVB}
\xymatrix{ F \ar[r]^{\wt{\varphi}}\ar[d]_{}& E\ar[d]^{}\\
T N\ar[r]_{{T\varphi}}& T M}
\end{equation}
Here the vertical maps are the anchors.
\end{definition}

Anchored vector bundles often arise as parts of more elaborate structures, such as Lie algebroids (see Section \ref{sec:lialg}) or Courant algebroids.

\begin{example}\label{ex:reg}
An anchored vector bundle with injective anchor is the same as a regular distribution.
\end{example}

\begin{example}
A \emph{bisubmersion} \cite{and:hol} is a manifold $Q$ with two surjective submersions
$\mathsf{s},\mathsf{t}\colon Q\to M$. Given a \emph{bisection}  $j\colon M\to Q$
 (that is, $\mathsf{s}\circ j=\mathsf{t}\circ j=\on{id}_M$),
the normal bundle $E=j^*(TQ)/TM$ has the structure of an anchored vector bundle, with the anchor induced by the difference $T\mathsf{s}-T\mathsf{t}\colon TQ\to TM$.
\end{example}

\begin{example}
Let $(E,\a)$ be an anchored vector bundle over $M$, and $F\subset E$ an anchored subbundle along a submanifold $N\subset M$. Then $\nu(E,F)$ is an anchored vector bundle over
$\nu(M,N)$, with anchor $\nu(\a)\colon \nu(E,F)\to \nu(TM,TN)\cong T\nu(M,N)$.
\end{example}

\begin{example}
Let $\dd$ and $V$ be vector spaces, and $\varrho\colon \dd\to \End(V)$ a linear map. Then $E=V\times \dd$ is an anchored vector
bundle with $\a(v,\zeta)=(v,\varrho(\zeta).v)\in TV$.  If $(E,\a)$ is an anchored vector bundle with $\a(E_m)=0$ at some point $m$, then its
\emph{linear approximation}  $\nu(E,E_m)$ at $m$ is of this type, with $V=T_mM,\ \dd=E_m$, and $\varrho$ the normal derivative of
the anchor, viewed as a section of the bundle $\on{Hom}(E,TM)$.
\end{example}

The group of automorphisms of an anchored vector bundle $(E,\a)$
will be denoted by $\on{Aut}_{AV}(E)$, and the Lie algebra of infinitesimal automorphisms by $\mf{aut}_{AV}(E)$.
Thus $\wt{X}\in \mf{aut}_{AV}(E)$ are the infinitesimal vector bundle automorphisms
satisfying
\[ \wt{X}\sim_{\a} X_T,\]
where $X_T\in\mf{X}(TM)$ is the  tangent lift of $X$. Equivalently, the corresponding operator $D$ on sections (cf.~ \eqref{eq:derprop}) satisfies
\begin{equation}\label{eq:anchaut}
 \a(D\tau)= [X,\a(\tau)]\end{equation}
for all $\tau\in\Gamma(E)$. The local flow defined by $\wt{X}\in \mf{aut}_{AV}(E)$
is by local automorphisms of the anchored vector bundle $(E,\a)$.
\begin{example}
Suppose that the anchor map $\a$ is injective, defining an inclusion $E\hra TM$. Then
$\a$ determines an isomorphism from $\mf{aut}_{AV}(E)\subset \mf{X}(E)$
to $\Gamma(E)\subset \mf{X}(M)$; the lift  $\wt{X}$ of $X\in \Gamma(E)$ is the
tangent lift $X_T\in \mf{X}(TM)$, restricted to $E\subset TM$.
\end{example}

\subsection{Pull-backs of anchored vector bundles}\label{subsec:pbAVB}
Suppose that $(E,\a)$ is an anchored vector bundle  over $M$, and  $\varphi\colon N\to M$ is a smooth map transverse to $\a$. Then the fiber product
\[
\xymatrix{ \varphi^! E \ar[r]^{}\ar[d]_{}& E\ar[d]^{\a}\\
T N\ar[r]_{{T\varphi}}& T M}
\]
defines an anchored vector bundle $\varphi^!E$ over $N$, such that the diagonal map  $\varphi^!E\to E\times TN$ is a morphism of anchored vector bundles.
The upper horizontal map is a morphism of anchored vector bundles, with base map $\varphi$.
Notable special cases include:
\begin{enumerate}
\item $\varphi^!TM=TN$;
\item if $N=M\times Q$, with $\varphi$ the projection to $M$, then
$\varphi^!E=E\times TQ$;
\item if $i\colon N\hra M$ is a submanifold transverse to $\a$, then $i^!E=\a^{-1}(TN)$;
\item if $\a$ is injective, so that $E\subset TM$,  then $\varphi^!E=(T\varphi)^{-1}(E)\subset TN$;
\item if $\varphi$ is a diffeomorphism, then $\varphi^!E=\varphi^*E$, the usual pull-back as a vector bundle.
\end{enumerate}
Under composition of maps, one has that $\psi^!(\varphi^! E) = (\varphi\circ \psi)^! E$, provided that the appropriate transversality conditions are satisfied.

\subsection{Transversals}
Let $(E,\a)$ be an anchored vector bundle over $M$.
\begin{definition}
A \emph{transversal} for $(E,\a)$ is a submanifold $i\colon N\hra M$ transverse to the anchor.
\end{definition}
Given a transversal, we can form the anchored vector bundle $i^!E=\a^{-1}(TN)$. Its pull-back
to the normal bundle $p\colon \nu_N\to N$ has the structure of a double vector bundle,
\[
\xymatrix{ p^!i^! E \ar[r]^{}\ar[d]_{}& i^! E\ar[d]\\
\nu_N \ar[r]_{}& N.}\]
Here, the vector bundle structure for the upper horizontal arrow is obtained by restriction from the
vector bundle structure on $i^!E\times T\nu_N\to i^!E\times TN$. In particular, the corresponding Euler vector field
is the restriction of $(0,\E_T)$ to $p^!i^!E\subset i^!E\times T\nu_N$.

The following Lemma shows that $p^!i^!E$ may be regarded as a linear approximation of $E$ along $N$.
%
\begin{lemma}\label{lem:fibprod}
Given a transversal $i\colon N\hra M$ for $(E,\a)$, there is a canonical isomorphism  of double vector bundles
\begin{equation}\label{eq:note}
\nu(E,\,i^!E) \cong p^!i^!E\end{equation}
intertwining the anchor maps.
\end{lemma}
\begin{proof}
The normal bundle functor, applied to
$\a\colon (E,i^!E)\to (TM,TN)$,  gives a commutative diagram
\begin{equation}\label{eq:first}
 \xymatrixcolsep{50pt}\xymatrix{
{\nu(E,i^!E)} \ar[r]^{p_{i^!E}}\ar[d]_{\nu(\a)}    & i^!E \ar[d]_{\a} \\
{\nu(TM,TN)}\ar[r]^{{p}_{TN}} & TN.
}
\end{equation}
It follows from transversality (see the comment after Equation \eqref{eq:mapofpairs}) that the left vertical map is a fiberwise vector bundle isomorphism, with base map the right vertical map. We conclude that \eqref{eq:first} is a fiber product diagram. By  \eqref{eq:tnu}, the lower left corner of the diagram can be replaced with $T\nu(M,N)$. Then the lower horizontal map becomes $Tp$, and the left vertical map an anchor map for $\nu(E,i^!E)$.  But the fiber product of $ T\nu(M,N)$ and $i^!E$ over $TN$ is exactly $p^!i^!E$, by definition. We conclude that
\begin{equation}\label{eq:19}
 \nu(E,i^!E)\to i^!E\times T\nu_N,\ \ \xi\mapsto \big(p_{i^!E}(\xi),\ \nu(\a)(\xi)\big)
\end{equation}
defines an injective morphism of double vector bundles
\[
{\xymatrix{ \nu(E,i^!E) \ar[r]^{}\ar[d]_{}& i^! E\ar[d]\\
\nu_N \ar[r]_{}& N}} \ \ \ \raisebox{-20pt}{$\longrightarrow$} \ \ \ {\xymatrix{i^! E\times T\nu_N \ar[r]^{}\ar[d]_{}& i^! E\times TN\ar[d]\\
N\times \nu_N \ar[r]_{}& N\times N}}
\]
with image the double vector bundle $p^!i^!E\subset i^!E\times T\nu_N$
\end{proof}

In the next sections, we formulate a condition under which $(E,\a)$ is isomorphic near $N$ to its linear approximation.
The proofs will involve the following fact.
\begin{lemma}\label{lem:epsilon}
Let $(E,\a)$ be an anchored vector bundle over $M$, and $N\subset M$ a transversal.
Then there exists a section $\epsilon \in\Gamma(E)$
with $\epsilon|_N=0$, such that  $X=\a(\epsilon)$ is Euler-like.
Given an action of a Lie group $G$ by automorphisms of $(E,\a)$, such that the action on the base is proper, one can take the section $\epsilon$ to be $G$-invariant.
\end{lemma}
\begin{proof}Consider the exact sequence
\begin{equation}\label{eq:eexact}
 0\to i^!E\to E|_N\to \nu_N\to 0,
\end{equation}
where the last map is the anchor map $E|_N\to TM|_N$ followed by the quotient map.
A bundle map $\nu_N\to E|_N$ defines a splitting of \eqref{eq:eexact} if and only if its composition with the anchor defines a splitting of
\begin{equation}\label{eq:texact}
 0\to TN\to TM|_N\to \nu_N\to 0.
 \end{equation}
Choose $\epsilon\in \Gamma(E)$ with $\epsilon|_N=0$, such that
the normal derivative of $\epsilon$ defines a splitting of \eqref{eq:eexact}. Then
$X=\a(\epsilon)$ satisfies $X|_N=0$, and since $\d^N X = \d^N \a(\epsilon)=\a(\d^N\epsilon)$, the normal derivative of $X$ defines a splitting of \eqref{eq:texact}. That is,
$\d^N X\colon \nu_N\to TM|_N$ followed by projection $TM|_N\to \nu_N$ is the identity.
By Remark~\ref{rem:linearapprox} (iii), the linear approximation $\nu(X)$ equals
\emph{minus} the normal derivative, $-\d^N X$, followed by the projection to $\nu_N$.
Thus $\nu(X)=-\on{id}=\kappa_{-1}$,  which agrees with $\E$ by Example~\ref{ex:euler}.
Multiplying $\epsilon$ by a bump function, we may arrange that $X=\a(\epsilon)$ is complete (see Remark \ref{rem:bump}).

In the $G$-equivariant setting, if the action on the base is proper,
choose a $G$-equivariant open cover consisting of flow-outs of slices for the action.
Over each slice, one can make $\epsilon$ invariant by averaging  (using that the stabilizer
groups are compact). This then extends to an invariant section on the flow-out of the slice.
Finally, one patches the local definitions by using a $G$-invariant partition of unity.
\end{proof}


\subsection{Normal form theorem}
One of several versions of the Stefan-Sussmann theorem asserts that if a smooth generalized distribution $D\subset TM$ is spanned by a locally finitely generated submodule $\mf{C}\subset \mf{X}(M)$,
such that $\mf{C}$ is closed under Lie brackets,  then $D$ defines a generalized foliation. Stefan-Sussmann \cite{ste:int,sus:orb} also gave integrability criteria in terms of the submodule $\mf{D}\subset \mf{X}(M)$ of \emph{all} vector fields tangent to $D$, but these contain errors; see Balan \cite{bal:not} for counter-examples and corrections. In the case of anchored vector bundles, we take $\mf{C}=\a(\Gamma(E))$.
\begin{definition}
An anchored vector bundle $(E,\a)$ will be called \emph{involutive} if $\a(\Gamma(E))$
is a Lie subalgebra of $\mf{X}(M)$.
\end{definition}

\begin{example}\label{ex:bracket}
Let $(E,\a)$ be an anchored vector bundle equipped with an additional map
$[\cdot,\cdot]\colon \Gamma(E)\times \Gamma(E) \to \Gamma(E)$ such that $\a([\sigma,\tau])=[\a(\sigma),\a(\tau)]$, for all $\sigma, \tau \in \Gamma(E)$.
Then $\a(\Gamma(E))$ is a Lie subalgebra, and hence $(E,\a)$ is involutive. This applies to Lie algebroids, Courant algebroids, and various kinds of Leibniz algebroids.
\end{example}

\begin{proposition}\label{prop:pullbackinvolutive}
Suppose that $(E,\a)$ is an anchored vector bundle over $M$, and $\varphi\colon N\to M$ is a smooth map transverse to $\a$. If $E$ is involutive, then the anchored vector bundle $\varphi^!E$ over $N$ is involutive.
\end{proposition}
\begin{proof}
By construction, $\varphi^!E$ is identified with the anchored subbundle
$F\subset E\times TN$ along the graph of $\varphi$, consisting of elements
of $(E\times TN)|_{\on{Gr}(\varphi)}$ whose image under the anchor
is tangent to $\on{Gr}(\varphi)$. Given two sections $\sigma_1,\sigma_2$
of $F$, choose extensions to sections $\tau_1,\tau_2\in \Gamma(E\times TN)$. Their images $Y_1,Y_2$ under the anchor are tangent to
$\on{Gr}(\varphi)$, and so is the Lie bracket $Y=[Y_1,Y_2]$. Since $E\times TN$ is involutive, $Y$ lifts to a section
$\tau\in\Gamma(E\times TN)$.  Its restriction $\sigma=\tau|_{\on{Gr}(\varphi)}$ is a
section
of $F$, satisfying $\a(\sigma)=[\a(\sigma_1),\a(\sigma_2)]$.
\end{proof}

 The main result of this section is the following normal form result, showing that in the involutive case $E$ is locally isomorphic near a given transversal
to its linear approximation.
\begin{theorem}[Transversals for anchored vector bundles]\label{th:anchored}
Let $(E,\a)$ be an anchored vector bundle over $M$, and $N\subset M$ a transversal.
%
\begin{enumerate}
\item \label{it:tildex}
Suppose $(E,\a)$ is involutive. Then there exists  $\wt{X}\in\aut_{AV}(E)$ vanishing along $i^!E$, such that the base vector field $X$ is Euler-like along $N$.
\item \label{it:family}
Any $\wt{X}\in\aut_{AV}(E)$ as in \eqref{it:tildex} determines an isomorphism of anchored vector bundles
\[ \wt{\psi}\colon p^!i^!E\to E|_U,\]
with base map a tubular neighborhood embedding $\psi\colon \nu_N\to U\subset M$.
\end{enumerate}
If a Lie group $G$ acts on $(E,\a)$ by automorphisms, such that the action on the base $M$ is proper and preserves $N$, then $\wt{X}$ in \eqref{it:tildex}
can be chosen $G$-invariant, and for any such $\wt{X}$ the resulting map $\wt{\psi}$ in \eqref{it:family} is $G$-equivariant.
\end{theorem}
The proof will be given in Section \ref{subsec:proof}, but here is an outline. Using that $N$ is transverse to the anchor,  we may choose a section $\epsilon \in \Gamma(E)$ such that $X=\a(\epsilon)$ is Euler-like. In Section \ref{subsec:lift}, we show that the involutivity of $(E,\a)$ implies the existence of a lift $\wt{\a}\colon \Gamma(E)\to \aut_{AV}(E)$ of the anchor map; we define $\wt{X}=\wt{\a}(\epsilon)$. We then argue that the vector field $\wt{X}$ on the total space $E$ is Euler-like  along $i^!E$. The map $\wt{\psi}$ is then obtained as a tubular neighborhood embedding, after identifying $\nu(E,i^!E)=p^!i^!E$.

If the normal bundle is trivial, the normal form in Theorem \ref{th:anchored} simplifies: For any choice of trivialization $\nu_N = N\times \P$ one gets
\[ p^!i^!E=i^!E\times T\P\]
as anchored vector bundles.  As a special case we obtain:
\begin{corollary}[Local splitting of anchored vector bundles]\label{cor:model}
Let $(E,\a)$ be an involutive anchored vector bundle over $M$. Let $m\in M$, and
$N\subset M$ a submanifold containing $m$, with $\a(E_m)\oplus T_mN=T_mM$. Let $\P=\a(E_m)$. Then $E$ is isomorphic, near $m$, to the direct product
\[
i^!E\times T\P.
\]
If a compact Lie group $G$ acts by automorphisms of $(E,\a)$, such that
the action on $M$ fixes $m$ and preserves $N$, then one can take this isomorphism to be $G$-equivariant.
\end{corollary}

In particular, one obtains a version of the Stefan Sussmann-theorem: the generalized distribution $D=\a(E)\subset TM$ defined by the involutive anchored vector bundle $(E,\a)$ is integrable.
Indeed, in the local model of Corollary \ref{cor:model}
it is immediate that $\{m\}\times \P$ is a leaf of the distribution. More generally, the leaves of the singular foliation are seen to have a local product form $L\times \P$, where $L$ is a leaf of
$\a(i^!E)\subset TN$. For regular distributions  (see Example~\ref{ex:reg}) one recovers the (local) Frobenius theorem.

\begin{remark}
The standard proof of  the Stefan-Sussmann theorem constructs the leaves
by induction, using the flows of vector fields spanning $D\subset TM$.   See \cite[Proposition 1.12]{and:hol} or \cite[Section 1.5]{duf:po}. While our approach is not shorter, it has better functorial properties and, being coordinate-free, seems better suited to infinite-dimensional generalizations.
Indeed, for regular distributions  our argument is similar to the proof of Frobenius' theorem for
Banach manifolds in \cite[Chapter VI]{lan:dif} and \cite[Section 4.4]{ab:ma}.
\end{remark}

\subsection{Lift of the anchor map}\label{subsec:lift}
The space of sections $\sigma\in\Gamma(E)$ such that the vector field $\a(\sigma)$ lifts to $\aut_{AV}(E)$, is a $C^\infty(M)$-submodule of $\Gamma(E)$:
\begin{lemma}
Suppose that $X=\a(\sigma)$ lifts to $\wt{X}\in \aut_{AV}(E)$ with corresponding operator
$D=D_\sigma$, and let $f\in C^\infty(M)$. Then
\begin{equation}
 D_{f\sigma}\tau:=f D_\sigma\tau-(\a(\tau)f)\sigma
\end{equation}
defines a lift of $fX=\a(f\sigma)$ to $\wt{fX}\in \aut_{AV}(E)$.
\end{lemma}
\begin{proof}
By an easy computation, one verifies that $D_{f\sigma}$ satisfies the  derivation property
\eqref{eq:derprop} and is compatible with the anchor \eqref{eq:anchaut}.
\end{proof}
The involutivity of an anchored vector bundle $(E,\a)$ is equivalent to the existence of lifts
of $\a(\sigma)$ for \emph{all} sections $\sigma$.

\begin{proposition}\label{prop:invol}
The anchored vector bundle $(E,\a)$ is involutive if and only if the anchor map $\a\colon \Gamma(E)\to \mf{X}(M)$
admits a lift $\wt{\a}\colon \Gamma(E)\to \aut_{AV}(E)$, such that the following diagram commutes:
\begin{equation}\label{eq:lift}
\xymatrix{ & \aut_{AV}(E)\ar[d]\\ \Gamma(E)\ar[ru]^{\wt{\a}}\ar[r]_{\a} & \mf{X}(M).}
\end{equation}
In this case, one can arrange that the operators $D_\sigma\colon \Gamma(E)\to \Gamma(E)$
defined by the lifts $\wt{\a}(\sigma)$ satisfy
\begin{equation}\label{eq:fsig1}
 D_{f\sigma}\tau=f D_\sigma\tau-(\a(\tau)f)\sigma
\end{equation}
for all $f\in C^\infty(M),\ \sigma,\tau \in\Gamma(E)$. Given an action of a Lie group $G$  by automorphisms of $(E,\a)$, such that the action on the base $M$ is proper,
one can take the lift $\wt{\a}$ to be $G$-equivariant.
\end{proposition}
\begin{proof}
It is convenient to consider the notion of an \emph{$\a$-connection},
given by a bilinear map
\[ \nabla\colon \Gamma(E)\times \Gamma(E)\to \Gamma(E),\ \ (\sigma,\tau)\mapsto \nabla_\sigma\tau\]
that is $C^\infty(M)$-linear in $\sigma$ and satisfies
\[ \nabla_\sigma(f\tau)=f\nabla_\sigma(\tau)+(\a(\sigma)f)\tau\]
for $f\in C^\infty(M),\ \ \sigma,\tau\in \Gamma(E)$.  (An ordinary vector bundle connection $\nabla'$ on
$E$ defines an $\a$-connection, by setting $\nabla_\sigma=\nabla'_{\a(\sigma)}$.) We define the \emph{torsion tensor} $\ca{T}_\nabla\in \Gamma(\wedge^2 E^*\otimes TM)$
by
\[  \ca{T}_\nabla(\sigma,\tau)
=\a(\nabla_\sigma\tau)-\a(\nabla_\tau\sigma)-[\a(\sigma),\a(\tau)].
\]

Suppose that $\a(\Gamma(E))$ is a Lie subalgebra. Then the last term in the formula for $\ca{T}_\nabla(\sigma,\tau)$ lifts to a section of $E$, and hence
$\ca{T}_{\nabla}$ lifts to a tensor $S\in \Gamma(\wedge^2 E^*\otimes E)$. (One first defines the lift locally,
using a basis of sections, and then uses a partition of unity.) The new $\a$-connection
$\ol{\nabla}_\sigma\tau=\nabla_\sigma\tau-\hh S(\sigma,\tau)$ has vanishing torsion. But for any torsion-free $\a$-connection $\nabla$, the
formula
\[ D_\sigma\tau=\nabla_\sigma\tau-\nabla_\tau\sigma\]
has the property
\eqref{eq:fsig1}, and
defines $\wt{\a}(\sigma)\in\aut_{AV}(E)$ lifting the vector field $\a(\sigma)$. In the presence of a $G$-action on $(E,\a)$ for which the action on the base is proper, one may take $\nabla$ to be $G$-equivariant, resulting in a $G$-equivariant lift $\wt{\a}$.

Conversely, given the lift $\wt{\a}$, with corresponding operators $D_\sigma$ on sections, we have that $\a(D_\sigma\tau)=[\a(\sigma),\a(\tau)]$. Hence $\a(\Gamma(E))$ is a Lie subalgebra.
\end{proof}

\begin{remark}
The lift $\wt{\a}$ in \eqref{eq:lift} is not a $C^\infty(M)$-module homomorphism, in general.
For example, if $E=TM$, with $\a$ the identity map, one can take $\wt{\a}$ to be the tangent lift of vector fields.
More generally, for Lie algebroids (treated in the next section) there is a natural lift \eqref{eq:lalift}
defined by the  bracket. These illustrate lifts which are not $C^\infty(M)$-linear.
\end{remark}

\subsection{Proof of the normal form theorem }\label{subsec:proof}
\begin{proof}[Proof of Theorem \ref{th:anchored}]
\begin{enumerate}
\item
By Lemma \ref{lem:epsilon}, we may choose
$\epsilon\in \Gamma(E)$ vanishing on $N$ and such that $X=\a(\epsilon)$ is Euler-like.
By Proposition \ref{prop:invol}, since $(E,\a)$ is involutive, there exists a lift $\wt{\a}\colon \Gamma(E)\to \aut_{AV}(E)$ satisfying \eqref{eq:fsig1}. Put $\wt{X}=\wt{\a}(\epsilon)\in\aut_{AV}(E)$. We claim
that $\wt{X}$ vanishes along $i^!E=\a^{-1}(TN)$, as a consequence of property  \eqref{eq:fsig1}. Indeed, the condition $\wt{X}|_{\a^{-1}(TN)} =  0$ is equivalent to saying that the flow of $\wt{X}$ restricts to the identity map on the subbundle $\a^{-1}(TN)\to N$ of $E\to M$. In terms of the operator $D_\epsilon: \Gamma(E)\to \Gamma(E)$ corresponding to $\wt{X}$, this  translates into the following condition: for all $\tau\in \Gamma(E)$ such that $\tau|_N \in \Gamma(\a^{-1}(TN))$, we have $D_\epsilon(\tau)|_N=0$. Since $\epsilon|_N=0$, we can assume that $\epsilon$ is of the form $f \sigma$, where $\sigma\in \Gamma(E)$ and $f\in C^\infty(M)$ vanishes on $N$. Then
$$
D_\epsilon(\tau) = D_{f\sigma}(\tau) = f D_\sigma(\tau) - (\a(\tau)f)\sigma
$$
implies that $D_\epsilon(\tau)|_N$ vanishes  when $\tau|_N \in \Gamma(\a^{-1}(TN))$.

In the $G$-equivariant situation, assuming that the action on $M$ is proper, one may take the section $\epsilon$ to be $G$-equivariant (cf. Lemma \ref{lem:epsilon}), and similarly for the  lift $\wt{\a}$. It then follows that $\wt{X}$ is
$G$-invariant.

\item
Suppose that $\wt{X}\in\aut_{AV}(E)$  vanishes along $i^!E\subset E$, and is a lift of an  Euler-like vector field $X\in\mf{X}(M)$.
Since $\wt{X}$ is $\a$-related to the tangent lift  $X_T$, the linear approximations are related under the bundle map  $\nu(\a)$ in \eqref{eq:first}:
\[ \nu(\wt{X})\sim_{\nu(\a)} \nu(X_T).\]
By Example \ref{ex:taneuler}, the tangent lift $X_T$ is Euler-like, thus $\nu(X_T)$ is the Euler vector field of $\nu(TM,TN)$.  Since the bundle map $\nu(\a)$ is a fiberwise isomorphism,  it follows that $ \nu(\wt{X})$ is the Euler vector field
for $\nu(E,i^!E)$. That is, $\wt{X}$ is Euler-like.
Let $\Phi_s$ be the flow of $X$, and $\wt{\Phi}_s$ the flow of $\wt{X}$. Write
$\lambda_t=\Phi_{{-\log(t)}}$, so that $\lambda_t\circ \psi=\psi\circ \kappa_t$, where $\psi: \nu_N\to U\subseteq M$ is the tubular neighborhood embedding determined by $X$.
 For all $t>0$, the map
$\wt{\lambda}_t=\wt{\Phi}_{{-\log(t)}}$ restricts to an automorphism of the anchored vector bundle $E|_U$,
with base map the restriction $\lambda_t|_U$. Since $\wt{X}$ is Euler-like, this family extends smoothly to $t=0$. Since $\wt{\lambda}_t$ preserves anchors for all $t>0$, the same is true for the limit $t=0$. Hence, $\wt{\lambda}_0\colon E|_U\to E|_U$ is a morphism of anchored vector bundles projecting onto
$i^!E\subset E|_U$. The morphism of anchored vector bundles
\begin{equation}\label{eq:direct}
 E|_U\to i^!E\times T\nu_N,\ \ \xi\mapsto \Big(\wt{\lambda}_0(\xi),\ T\psi^{-1}(\a(\xi))\Big),
\end{equation}
is injective, and since
\[ \a(\wt{\lambda}_0(\xi))=T\lambda_0(\a(\xi))=Tp(T\psi^{-1}(\a(\xi))),\]
we find that its image is exactly $p^!i^!E$. We take $\wt{\psi}\colon p^!i^!E\to E|_U$ to be the inverse map. In the $G$-equivariant case, if the vector field $\wt{X}$ is $G$-invariant, then all maps in this construction are $G$-equivariant, hence so is $\wt{\psi}$.
\end{enumerate}
\end{proof}

\begin{remark}
Note that \eqref{eq:direct} relates $\wt{X}$ with the vector field  $(0,X_T)$ on $i^!E\times T\nu_N$, which therefore restricts to $p^!i^!E$. In turn, the isomorphism $\nu(E,i^!E)\cong p^!i^!E$ from Lemma \ref{lem:fibprod} intertwines  this vector field on $p^!i^!E$ with the Euler vector field for $\nu(E,i^!E)$.
(See Equation \eqref{eq:19}.)  We conclude that the isomorphism $\wt{\psi}\colon  p^!i^!E\cong \nu(E,i^!E)\to E|_U$ takes
the Euler vector field to $\wt{X}$. It hence follows from the uniqueness part in
Proposition \ref{prop:euler} that $\wt{\psi}$ is exactly the tubular neighborhood embedding defined by $\wt{X}$.
\end{remark}

\begin{remark}\label{rem:directdesc}
The maps \eqref{eq:direct} generalize to injective morphisms of anchored vector bundles, for all $t\ge 0$,
\begin{equation}\label{eq:fam1a}
E|_U\to E|_U\times T\nu_N,\ \
\xi\mapsto \big(\wt{\lambda}_t(\xi),T\psi^{-1}(\a(\xi))\big).
\end{equation}
Since $\a(\wt{\lambda}_t(\xi))=T\lambda_t(\a(\xi))=T(\lambda_t\circ \psi)(T\psi^{-1}(\a(\xi))$, we see that the image of
\eqref{eq:fam1a} is $\psi^! \lambda_t^!(E|_U)=\kappa_t^!\psi^!E$. Let
\begin{equation}\label{eq:wtpsit}
\wt{\psi}_t\colon \kappa_t^!\psi^!E\to E|_U
\end{equation}
be the inverse map. Note that $\wt{\psi}_t$ has base map $\psi$, for all $t\ge 0$. For $t=0$ it is the isomorphism $\wt{\psi}$ constructed
above, noting that $\psi\circ \kappa_t=i\circ p$.
 For $t>0$, it can be described as the `obvious' isomorphism $\kappa_t^!\psi^!E\to E|_U$ (with base $\psi\circ \kappa_t=\lambda_t\circ \psi
\colon \nu_N\to U$), given by the ordinary vector bundle pullback with respect to the diffeomorphism
$\psi\circ \kappa_t=\lambda_t\circ \psi\colon \nu_N\to U$,
followed by the inverse of the map $\wt{\lambda}_t\colon E|_U\to E_U$ (with base map the inverse of $\lambda_t$).
\end{remark}

\subsection{Functorial properties}\label{subsec:funav}
Let $\Phi\colon (M',N')\to (M,N)$ be a map of pairs, lifting to a morphism of anchored vector bundles $\wt{\Phi}\colon E'\to E$. Suppose that the anchor maps of
$E',E$ are transverse to $i'\colon N'\to M',\ i\colon N\to M$, respectively. The map $\wt{\Phi}$ restricts to a morphism of anchored vector bundles $i'^!E'\to i^!E$, giving rise
to a morphism of anchored vector bundles
\begin{equation}\label{eq:nuPhi}
 \nu(\wt{\Phi})\colon \nu(E',i'^!E')\to \nu(E,i^!E)\end{equation}
with base map $\nu(\Phi)\colon \nu(M',N')\to \nu(M,N)$. On the other hand, the
map $i'^!E'\times \nu_{N'}\to i^!E\times \nu_N$ restricts to a morphism $p'^!i'^!E'\to p^!i^!E$, which coincides with \eqref{eq:nuPhi} under the isomorphism from Lemma \ref{lem:fibprod}. Suppose now that $\wt{X}'\in\aut_{AV}(E')$ is as in the theorem, with $\wt{X}'\sim_{\wt{\Phi}} \wt{X}$.  Then the corresponding tubular neighborhood embeddings give a commutative diagram
\begin{equation}\label{diag:funct2} \xymatrixcolsep{50pt}
\xymatrix{{ {p'}^!{i'}^!E' \cong \nu(E',{i'}^!E')} \ar[r]\ar[d]_{\nu(\wt{\Phi})}    & E'\ar[d]^{\wt{\Phi}}\\{p^!i^!E \cong \nu(E,i^!E)}\ar[r] & E}\end{equation}
where all maps are morphisms of anchored vector bundles.
A similar functorial property holds relative to \emph{co}morphisms of anchored vector bundles. Recall that a {\em comorphism} of (anchored) vector bundles, with base map  $\Phi\colon M'\to M$  is given by a bundle map $\Phi^*E\to E'$ whose graph in $E\times E'$ is an (anchored) subbundle along the graph $\on{Gr}(\Phi)\subset M\times M'$. We write $\wt{\Phi}\colon E'\da E$ for such a `wrong-way' morphism, and $\on{Gr}(\wt{\Phi})\subset E\times E'$ for its graph.  By a discussion similar to that for morphisms,
 one obtains comorphisms $p'^!i'^!E'\da p^!i^!E$, and in the case of
$\wt{X}'\sim_{\wt{\Phi}} \wt{X}$ there is a commutative diagram
\begin{equation}\label{diag:funct4} \xymatrixcolsep{50pt}
\xymatrix{{ {p'}^!{i'}^!E' \cong \nu(E',{i'}^!E')} \ar[r]\ar@{-->}[d]_{\nu(\wt{\Phi})}    & E'\ar@{-->}[d]^{\wt{\Phi}}\\{p^!i^!E \cong \nu(E,i^!E)}\ar[r] & E.}\end{equation}
In fact, one can consider the result for comorphisms as a special case of the result for morphisms, applied to the inclusion map $\on{Gr}(\wt{\Phi})\hra E\times E'$. Observe that if
$\wt{\Phi}\colon E'\to E$ is a comorphism, and $N\subset M$ is a transversal for $E$ which is also transverse to the map $\Phi$, then it is automatic that $N':=\Phi^{-1}(N)$ is a transversal for $E'$.
To see this, let $v'\in TM'|_{N'}$ be given, and let $v=T\Phi(v')\in TM|_N$ its image.
Write $v=v_1+v_2$ where $v_1\in TN$ and $v_2=\a(\xi),\ \xi\in E|_N$.
Let $\xi'$ be the image of $\Phi^*\xi$ under the map $\Phi^*E\to E'$,  with the same base point as $v'$. Then $T\Phi(\a(\xi'))=\a(\xi)$. Putting $v_2'=\a(\xi')$, it follows that $v_1':=v'-v_2'$ is tangent to $N'$.

\subsection{Uniqueness of transverse structures}\label{subsec:uni1}
Let $\psi\colon N\to Q$ be a submersion, with fibers $N_q=\psi^{-1}(q)$.
A \emph{family of anchored vector bundles} $F_q\to N_q$ is an anchored vector bundle $F\to N$ whose anchor is tangent to $\ker(T\psi)$, with $F_q=F|_{N_q}$ the restriction. We will call such a family \emph{infinitesimally trivial} if every
$Z\in  \mf{X}(Q)$ admits a lift $Y\in\mf{X}(N)$ (i.e.,  $Y\sim_\psi Z$) which is
the base vector field of an infinitesimal automorphism
$\wt{Y}\in \aut_{AV}(F)$. Note that in this case, the (local) flow of $Y$ preserves
the fibers of $\psi$, and the (local) flow of $\wt{Y}$ is by (local) isomorphisms of anchored vector bundles.

We interested in the following situation. Suppose that $(E,\a)$ is an anchored vector bundle over $M$, and $\phi\colon N\to M$ a smooth map such that all $i_q=\phi|_{N_q}\colon N_q\to M$  are transversals,  i.e., transverse to the anchor map of $E$. Then $F_q=i_q^!E$ is a family of anchored vector bundles. Here $F\subset \phi^!E$ is the subbundle given as the pre-image of
$\ker(T\psi)$ under the anchor.
\begin{proposition}\label{prop:uniq}
Suppose that $(E,\a)$ is an involutive anchored vector bundle over $M$, and that $i_q\colon N_q\subset M$ is a smooth family of transversals, as above.
 Then the family of anchored vector bundles
$i_q^!E$ is infinitesimally trivial.
\end{proposition}
\begin{proof} By Proposition \ref{prop:pullbackinvolutive}, the bundle $\phi^!E$ is involutive. We claim that the fibers $N_q\subset N$ are transversals for $\phi^!E$. Indeed, given $y\in N$, with image $x=\phi(y)$, and any $w\in T_yN$,
the image $v=T\phi(w)$ can be written as a sum $v_1+v_2$, where
$v_1\in \a(E)_x$ and $v_2\in T_x(i_q(N_q))$. Let $w_2\in T_y N_q$
be the pre-image. Then $w_1=w-w_2$ satisfies $T\phi(w_1)=v-v_2=v_1$, hence
$w_1\in \a(\phi^!E)_y$, proving the claim.

The transversality implies that the bundle map $T\psi\circ \a_{\phi^!E}\colon \phi^!E\to TQ$ is fiberwise onto. Hence, for any given $Z\in\mf{X}(Q)$ there is a section  $\sigma\in \Gamma(\phi^!E)$ such that its image under the anchor, denoted by $Y$, satisfies  $Y\sim_\psi Z$. Furthermore, by Proposition \ref{prop:invol} it admits a lift
 $\wt{Y}\in \aut_{AV}(\phi^!E)$. Since $\wt{Y}$ is related under the anchor map to the
 tangent lift $Y_T$, and the latter is tangent to $\ker(T\psi)$ (due to $Y\sim_\psi Z$), it is automatic that $\wt{Y}$ is tangent to $F
 \subset \phi^!E$. Hence, $\wt{Y}$ restricts to an element of $\aut_{AV}(F)$.
\end{proof}
As a special case, we obtain:
\begin{corollary}\label{cor:uniqlocal}
Let $(E,\a)$ be an involutive anchored vector bundle over $M$, $S$ a leaf, and $m_0, m_1 \in S$. Let $N_0$ and $N_1$ be transversals of dimension $\dim M-\dim S$, with inclusions $i_0$ and $i_1$, and such that $N_0\cap S=\{m_0\}$ and $N_1\cap S=\{m_1\}$. Then, after replacing the $N_i$ with smaller neighborhoods of $m_i\in N_i$ if necessary,  there is an isomorphism of the induced anchored vector bundles $i_0^! E \to i_1^! E$, taking $m_0$ to $m_1$.
\end{corollary}
\begin{proof}
Given a smooth path $\R\to S,\ s\mapsto m_s$, taking on the given values at $s=0,1$, one can find a family of transversals $i_s\colon N_s\to M$ with $N_s\cap S=\{m_s\}$ for all $s$. That is, the union of $N_s\subset M\times\{s\}$ defines a submanifold $N\subset M\times \R$. The maps $\phi\colon N\to M$ and $\psi\colon N\to \R$ are given by projections to the two factors.  By Proposition \ref{prop:uniq}, or rather its proof, there exists a $\wt{Y}\in \aut_{AV}(\phi^!E)$ such that the base vector field $Y$ is tangent to the leaves of $\phi^!E$ and satisfies $Y:=\a(\sigma)\sim_\psi \f{\p}{\p s}$. As argued above, $\wt{Y}$  preserves $F\subset \phi^!E$, where $F_s=i_s^!E$.

The path $s\mapsto m_s\in N_s\cap S$ defines a section $\R\to N$ of the submersion $\psi$. By construction, this is a single leaf of $\phi^!E$.
Since $Y$ is tangent to the leaves,  its restriction to $\R\subset N$ is just
$\f{\p}{\p s}$. In particular, there exists a neighborhood of $m_0$ in $N_0$ over which the  flow of $Y$ (and hence also of $\wt{Y}$) is defined for time $1$. The time 1-flow of
$\wt{Y}$ gives the desired isomorphism of anchored vector bundles $i_0^!E\to i_1^!E$ over possibly smaller neighborhoods of $m_i$ in $N_i$.
\end{proof}

\section{Lie algebroids}\label{sec:lialg}
Suppose that $(E,\a,[\cdot,\cdot])$ is a Lie algebroid over $M$. Thus $(E,\a)$ is an anchored vector bundle, with a Lie bracket
$[\cdot,\cdot]$ on its space $\Gamma(E)$ of sections satisfying the compatibility property
\[ [\sigma,f\tau]=f[\sigma,\tau]+(\a(\sigma)f)\, \tau,\]
for $\sig,\tau\in\Gamma(E)$ and $f\in C^\infty(M)$.
As is well-known, this implies that $\a\colon \Gamma(E)\to \mf{X}(M)$ preserves Lie brackets.
We denote by $\on{Aut}_{LA}(E)$ the automorphisms of $E$ preserving the Lie algebroid
structure, and by $\aut_{LA}(E)$ the infinitesimal automorphisms, consisting of all $\wt{X}\in \aut_{AV}(E)$ such that the corresponding operator $D$ on sections is a derivation of the Lie bracket. In particular, the operators $D_{\sigma}=[\sigma,\cdot]$ define infinitesimal automorphisms $\wt{\a}(\sigma)\in \aut_{LA}(E)$. The resulting lift
\begin{equation}\label{eq:lalift}
\wt{\a}\colon \Gamma(E)\to \aut_{LA}(E).
\end{equation}
has the property \eqref{eq:fsig1}.

\subsection{Normal form theorem}
Given a smooth map $\varphi\colon N\to M$ transverse to $\a$, the anchored vector bundle $\varphi^!E$ over $N$ inherits a unique Lie algebroid structure, in such a way that the diagonal map $\varphi^!E\to E\times TN$ is an inclusion as a Lie subalgebroid.  See \cite[Section 4.3]{mac:gen} or \cite{lib:cou}. For any transversal  $i\colon N\hra M$,  with normal bundle $p\colon \nu_N\to N$, we obtain a Lie algebroid $p^!i^!E\to \nu_N$.
Theorem \ref{th:anchored} has the following refinement:
\begin{theorem}\label{th:liealgebroids}
Let $(E,\a,[\cdot,\cdot])$ be a Lie algebroid over $M$, and let $N\subset M$ be a transversal. Choose $\epsilon\in \Gamma(E)$ with $\eps|_N=0$, such that $X=\a(\epsilon)$ is Euler-like. The choice of $\epsilon$ determines a tubular neighborhood embedding
$\psi\colon \nu_N\to M$ with an  isomorphism of Lie algebroids
\[ \wt{\psi}\colon p^!i^!E\to E|_U.\]
Given a $G$-action by Lie algebroid automorphisms of $E$, and a $G$-equivariant
choice of  $\epsilon$, the isomorphism $\wt{\psi}$ is $G$-equivariant.
\end{theorem}

\begin{proof}
 We use the same construction as in the proof of Theorem \ref{th:anchored}, but with the distinguished lift  \eqref{eq:lalift}. As discussed in Remark \ref{rem:directdesc}, the vector field $\wt{X}=\wt{\a}(\epsilon)$ determines a family of isomorphisms of anchored vector bundles $\wt{\psi}_t\colon \kappa_t^!\psi^!E\to E|_U$, for all $t\ge 0$.  For all $t>0$, these are given by the Lie algebroid automorphisms $\wt{\lambda}_t$,
and in particular preserve Lie brackets on sections. Hence, by continuity the map $\wt{\psi}_0$ preserves Lie brackets as well.
\end{proof}
If the normal bundle is trivial, $\nu_N=N\times \P$, then we obtain the simpler model
\[ p^!i^!E=i^!E\times T\P\]
\emph{as Lie algebroids}. In particular, we obtain:
\begin{corollary}[Local splitting of Lie algebroids]\label{cor:lialgebroids}
Let $(E,\a,[\cdot,\cdot])$ be a Lie algebroid over $M$, and $m\in M$. Let $i\colon N\hra  M$ be a submanifold containing $m$, such that $T_mN$ is a complement to $\P=\a_m(E_m)$ in
$T_mM$.  Then Lie algebroid $E$ is isomorphic, near $m$, to the direct product of Lie algebroids
$i^!E\times T\P$. If a compact Lie group $G$ acts on $E$ by Lie algebroid automorphisms, such that the action on $M$ fixes $m$ and preserves $N$, this isomorphism can be chosen $G$-equivariant.
\end{corollary}
For $G=\{1\}$ this result is due to Weinstein \cite{wei:alm}, Fernandes \cite{fer:lie}, and Dufour \cite{duf:nor}.

\subsection{Functorial properties}\label{subsec:funla}
The functorial properties of the construction are analogous to those for anchored vector bundles. Of particular interest is the functoriality with respect to Lie algebroid \emph{co}morphisms $\wt{\Phi}\colon E'\dasharrow  E$.

\begin{example}
Recall that any Poisson structure $\pi$ on $M$ makes $T^*M$ into a Lie algebroid, with anchor $\pi^\sharp\colon T^*M\to T^*M'$. Any Poisson map $M'\to M$ defines a comorphism of Lie algebroids $T^*M'\dasharrow T^*M$.
\end{example}
As remarked in Section \ref{subsec:funav}, if $N\subset M$ is a transversal for
$E$ which is also transverse to the map $\Phi$, then its pre-image $N'=\Phi^{-1}(N)$ is transversal for $E'$. Furthermore, if
$\epsilon\in \Gamma(E)$ has the properties in Theorem \ref{th:liealgebroids}, then
its image $\epsilon'$ under the map $\Gamma(E)\to \Gamma(E')$ determined by the comorphism has similar properties, with respect to $E'$.
Hence, in this situation \eqref{diag:funct4} is a commutative diagram of Lie algebroid comorphisms.

\subsection{Uniqueness of transverse structures}\label{subsec:uni2}
Just as in Section \ref{subsec:uni1}, we can consider families of Lie algebroids
$F_q\to N_q$, where $N_q$ are the fibers of a submersion $\psi\colon N\to Q$. For the  \emph{infinitesimal triviality} of such a family, one requires that for every $X\in\mf{X}(Q)$ there exists $\wt{Y}\in \aut_{LA}(E)$ such that $Y\sim_\psi X$. The same argument as in
Section \ref{subsec:uni1} shows that if $i_q\colon N_q\to M$ is a family of transversals for a Lie algebroid $(E,\a)$ then the family of Lie algebroids $i_q^!E$ is infinitesimally trivial. As a consequence, one obtains the natural analogue of Corollary~\ref{cor:uniqlocal} for Lie algebroids, which recovers  \cite[Thm.~1.2]{fer:lie}.

\subsection{Lie groupoids}
Suppose that $\G\rra M$ is a Lie groupoid, with source and target maps $\mathsf{s},\mathsf{t}\colon \G\to M$. The anchor $\a$ of its
Lie algebroid $E:=\nu(\G,M)$ is induced by $T\mathsf{s}-T\mathsf{t}\colon T\G\to TM$.
Using that $\ker (T_g\mathsf{t})$ (the tangent space to the $\mathsf{t}$-fiber at $g\in\G$)
is spanned by the left-invariant vector fields,
one sees that
\[ T_g\mathsf{s}(\ker (T_g\mathsf{t}))=\on{ran}(\a_{\mathsf{t}(g)}).\]
Hence, a smooth map
$\varphi\colon N\to M$ is transverse to $\a$ if and only if it is transverse to the restriction of $\mathsf{s}$ to every
$\mathsf{t}$-fiber. (Equivalently, it is transverse to the restriction of $\mathsf{t}$ to every
$\mathsf{s}$-fiber.)
In this case, there is a well-defined \emph{pull-back  Lie groupoid} \footnote{The condition is equivalent to transversality
of the maps $(\mathsf{t},\mathsf{s})\colon \G\to M\times M$ and $(\varphi,\varphi)\colon N\times N\to M\times M$;
hence $\varphi^!\G$ can be also regarded as a fibered product with respect to these two maps. See \cite[Appendix A]{bur:vec} for a general discussion, including a simplification of \cite[Proposition 2.3.1]{mac:gen}.}
%
\[ \varphi^!\G\rra N\]
where $\varphi^!\G=N\times_M \G\times_M N$ is the fiber product with respect to source and target maps. Here the transversality assumption ensures that the map $N\times_M \G\to M$ induced by the source map is transverse to $\varphi$; hence the second fiber product is well-defined. It also ensures that source and target for $\varphi^!\G$ are surjective submersions. The Lie algebroid of $\varphi^!\G$ is
$\varphi^!E$.

A \emph{transversal} for $\G$  is a submanifold $i\colon N\hra M$  such that $i$ is transverse to $\a$.  In this case,
$i^!\G=\G|_N=\mathsf{s}^{-1}(N)\cap \mathsf{t}^{-1}(N)$. Choose $\eps\in\Gamma(E)$ such that $X=\a(\eps)$ is Euler-like, defining a tubular neighborhood embedding $\psi\colon \nu_N\to U\subset M$.  The section $\eps$ defines a left-invariant vector field
$\eps^L\in\mf{X}(\G)$ tangent to the $\mathsf{t}$-fibers, and a right-invariant vector field $\eps^R\in \mf{X}(\G)$ tangent to the $\mathsf{s}$-fibers.
The difference
$\wt{X}=\eps^L-\eps^R$ is an infinitesimal Lie groupoid automorphism, related to
$X$ under both source and target maps, and is Euler-like along $i^!\G\subset \G$.  Similarly to Theorem \ref{th:liealgebroids}, using the flow of $\eps^L-\eps^R$
one obtains an isomorphism  of Lie groupoids
\[ \wt{\psi}\colon p^!i^!\G\xrightarrow{\cong} \G|_U\]
where  $\G|_U=\mathsf{s}^{-1}(U)\cap \mathsf{t}^{-1}(U)$.
In fact, one obtains a family $\wt{\psi}_t\colon \kappa_t^!\psi^!\G\to \G|_U$ of
groupoid isomorphisms, reducing to $\wt{\psi}$ at $t=0$. These are obtained as inverses of the maps
\[ \G|_U\xrightarrow{\cong}\kappa_t^!\psi^!\G\subset \G\times \nu_N\times \nu_N,\ \ g\mapsto \Big(\wt{\lambda}_t(g),\ \psi^{-1}\big(\mathsf{t}(g)\big),\ \psi^{-1}\big(\mathsf{s}(g)\big)\Big).\]

\section{Dirac manifolds}\label{sec:dirac}

We next obtain normal form theorems and splitting theorems for Dirac manifolds.
\subsection{Dirac structures}
We begin by recalling the definition of the standard \emph{Courant algebroid} $\T M$  over a manifold $M$, with possible twisting by a closed 3-form $\eta\in \Om^3(M)$.
References include Courant's original paper \cite{cou:di} as well as \cite{lib:cou,roy:co,sev:let,sev:poi}. We let
\[ \T M=TM\oplus T^*M\]
be equipped with the symmetric bilinear form
$\l v_1+\mu_1,v_2+\mu_2\r=\l\mu_1,v_2\r+\l\mu_2,v_1\r$,
the anchor $\a\colon \T M\to TM$ given by $v+\mu\mapsto v$,
and the \emph{Courant bracket} $\Cour{\cdot,\cdot}$ on its space of sections,
\[ \Cour{X_1+\alpha_1,X_2+\alpha_2}=[X_1,X_2]+\ca{L}_{X_1}\alpha_2-\iota_{X_2}\d \alpha_1
+\iota_{X_1}\iota_{X_2}\eta
\]
for vector fields $X_i$ and 1-forms $\alpha_i$.

A \emph{Dirac structure} on $M$ (relative to $\eta$)  is a subbundle $E\subset \T M$ such that $E=E^\perp$, and
such that the space of sections of  $E$ is closed under the bracket $\Cour{\cdot,\cdot}$. Dirac structures are Lie algebroids, with the anchor $\a_E$ and bracket $[\cdot,\cdot]$ obtained by restriction from $\T M$.
If $\eta=0$, then a Dirac structure with $E\cap TM=\{0\}$ is of the form $E=\on{Gr}(\pi)$, where
$\pi\in \Gamma(\wedge^2 TM)$ is a Poisson structure, and $\on{Gr}(\pi)$ is the graph of the bundle map $T^*M\to TM$ defined by $\pi$.
For any 2-form $\omega$ we define the \emph{B-field transform} $\ca{R}_\omega: \T M \to \T M$,
\[ \ca{R}_\omega(v+\mu)=v+\mu+\iota_v\omega,\]
and put $E^\omega=\ca{R}_\omega(E)$. Then  $E^\omega$ is a Dirac structure relative to the
3-form $\eta+\d\omega$. Given a smooth map $\varphi\colon N\to M$, define
\[ \varphi^!E = \{v'+\mu'\in \T N |\ \exists v+\mu\in E \colon v=T\varphi(v'),\ \mu'=\varphi^*\mu\}.\]
If $\varphi$ is transverse to $\a_E$, then $\varphi^!E\subset \T N$ is a Dirac structure relative  to $\varphi^*\eta$; as a Lie algebroid it coincides with the pull-back Lie algebroid discussed in Section \ref{sec:lialg}.  Given a 2-form $\omega$ on $M$, one finds that
\[
(\varphi^!E)^{\varphi^*\omega} = \varphi^!(E^{\omega}).
\]
If $\varphi$ is a diffeomorphism, and letting
\[ \T \varphi \colon  \T N \to \T M,\ \ \ v + \mu\mapsto T\varphi(v) + (T\varphi^{-1})^*(\mu) \]
(the sum of tangent and cotangent lifts), we have that $\varphi^!E=\T\varphi^{-1}(E)$.

\subsection{The normal form theorem}
A submanifold $i\colon N\hra M$ is called a \emph{transversal} for the Dirac structure $E$ if it is transverse to the
anchor of $E$. In this case, we obtain Dirac structures
$i^!E\subset \T N$ relative to $i^*\eta$ and $p^!i^!E\subset \T \nu_N$ relative to $p^*i^*\eta$.

\begin{theorem}[Normal form for Dirac stuctures]\label{th:dirac}
Let $E\subset \T M$ be a Dirac structure relative to $\eta$, and $i\colon N\hra M$ a transversal.  Choose $\epsilon=X+\alpha \in \Gamma(E)$  with $\epsilon|_N=0$, such that $X$ is Euler-like along $N$, and let $\psi\colon \nu_N\to U\subset M$ be the resulting tubular neighborhood embedding. Then $\T\psi\colon \T\nu_N\to \T M$ restricts to an isomorphism of Dirac structures
\[ (p^!i^!E)^{\omega}\to E|_U,\]
where $\omega\in \Om^2(\nu_N)$ is the 2-form
\begin{equation}\label{eq:w} \omega=\int_0^1 \f{1}{\tau}\kappa_\tau ^* \psi^*(\d\alpha+\iota_X\eta)\ \d \tau .\end{equation}
Given a proper $G$-action on $M$, preserving $\eta$ and such that its lift to $\T M$
preserves $E$, one can choose $\epsilon$ to be $G$-invariant. The resulting $\omega$ is then $G$-invariant, and the isomorphism $\T\psi$ is $G$-equivariant.
\end{theorem}
We will see that the 2-form $\omega$ is well-defined: the family of 2-forms
 $\f{1}{t}\kappa_t^*\psi^*(\d\alpha+\iota_X\eta) $ extends smoothly to $t=0$.
Furthermore,  $\d\omega=\psi^*\eta-p^*i^*\eta$.
The proof of Theorem \ref{th:dirac} will be given in Section \ref{subsec:proofdirac}; its functorial properties
will be discussed in Section \ref{subsec:diracfunc}.

\begin{remark} \label{rem:remark}
For later reference, we remark that Theorem \ref{th:dirac}, and its proof,  extend to \emph{complex Dirac structures} $E\subset \T_\C M$ inside the complexified Courant algebroid,
provided $\epsilon\in \Gamma(E)$ can be chosen in such a way that its vector field part is real, $X=\ol{X}$.
\end{remark}

\begin{remark}
Since $E$ is a Dirac structure, the Courant bracket restricts to a Lie-algebroid bracket
of $E$. Hence, the section $\epsilon$ defines an isomorphism of Lie algebroids
$p^!i^!E\to E|_U$, using the approach in Section \ref{sec:lialg}. The Theorem above gives a stronger  statement, since it treats $E$ not merely as a Lie algebroid, but as a Dirac structure embedded as a subbundle of $\T M$. Forgetting about this embedding, and identifying
$(p^!i^!E)^\omega$ with  $p^!i^!E$ \emph{as Lie algebroids}, one may verify that the isomorphism from Theorem
\ref{th:dirac} reduces to that of Theorem \ref{th:liealgebroids}.
\end{remark}

Theorem \ref{th:dirac} specializes to local splitting theorems for Dirac manifolds near given points $m\in M$.
\begin{corollary}\label{cor:dirac}
Let $E\subset \T M$ be a Dirac structure relative to the closed 3-form $\eta\in \Om^3(M)$, and $m\in M$. Let $N\subset M$ be a submanifold containing $m$, such that
$T_m N$ is a complement to $\P=\a_m(E_m)$ in $T_mM$. Then the Dirac structure $E$ is isomorphic, near $m$, to a Dirac structure of the form
\[ (i^!E\times T\P)^\omega\subset \T(N\times \P).\]
Here $\omega\in \Om^2(N\times \P)$ is a 2-form such that the (local) diffeomorphism of the base manifolds
$M$ and $N\times\P$ takes
$\eta$ to $p^*i^*\eta+\d\omega$. Given an action of a compact Lie group $G$ by Dirac automorphisms, fixing $m$ and preserving the submanifold $N$, one obtains a $G$-invariant $\omega$ and a $G$-equivariant isomorphism.
\end{corollary}
For the case $\eta=0,\ G=\{1\}$, this result is due to Blohmann \cite[Theorem 3.2]{blo:rem}.

 \subsection{Courant automorphisms}
 \label{subsec:couaut}
For the $\eta$-twisted Courant algebroid $\T M$,
we denote by $\on{Aut}_{CA}(\T M)$ the group of \emph{Courant automorphisms}, consisting of vector bundle automorphisms preserving the anchor, the symmetric pairing and the bracket.
The Lie algebra of infinitesimal Courant automorphisms $\wt{X}$ is denoted  by $\aut_{CA}(\T M)$; the corresponding operators $D$ on sections of $\T M$ preserve the anchor, as well as the bracket and pairing in the sense that
\[ \ca{L}_X \l \sigma_1,\ \sigma_2\r=\l D(\sigma_1),\ \sigma_2\r +\l \sigma_1,\ D(\sigma_2)\r,\]
\[  D\Cour{\sigma_1,\ \sigma_2}=\Cour{D(\sigma_1),\,\sigma_2}+\Cour{\sigma_1,\,D(\sigma_2)},
\]
for all $\sigma_i=X_i+\alpha_i\in \Gamma(\T M)$.
Any section $\sigma=X+\alpha\in \Gamma(\T M)$ defines an infinitesimal Courant automorphism
with $D=\Cour{\sigma,\cdot}$.

The group $\on{Aut}_{CA}(\T M)$ and its Lie algebra $\aut_{CA}(\T M)$ have the following explicit description. The B-field transforms $\ca{R}_\omega$ and the bundle maps $\T\Phi$ defined by diffeomorphisms $\Phi$ of $M$ combine into an injective group homomorphism
\begin{equation}\label{eq:semi}
 \Omega^2(M)\rtimes \on{Diff}(M)\to \Aut(\T M),\
(\omega,\Phi)\mapsto \ca{R}_{-\omega}\circ \T \Phi.\end{equation}
The image consists of vector bundle automorphisms preserving the anchor and the pairing. Similarly, there is a Lie algebra morphism  $\Omega^2(M)\rtimes \mf{X}(M)\to \aut(\T M)$ with the action $(\vartheta,X).(Y+\beta)=[X,Y]+\ca{L}_X\beta-\iota_Y\vartheta$.
The following is proved in  \cite[Proposition 2.5]{gua:ge1} and \cite[Section 2]{hu:ham}, using slightly different sign conventions.
\begin{proposition}
The  map \eqref{eq:semi} restricts to an isomorphism from the group of all $(\omega,\Phi)$ such that
\[ (\Phi^{-1})^*\eta-\d\omega=\eta\]
onto $\Aut_{CA}(\T M)$. Similarly, $\aut_{CA}(\T M)$ is isomorphic to the Lie subalgebra of $ \Omega^2(M)\rtimes \mf{X}(M)$ consisting of all $(\vartheta,X)$ such that $\ca{L}_X\eta+\d\vartheta=0$.
\end{proposition}
The formula for the Courant bracket shows that the infinitesimal Courant automorphism $\Cour{X+\alpha,\cdot}$ corresponds to $(\d\alpha+\iota_X\eta,\,X)\in \Om^2(M)\rtimes \mf{X}(M)$.
We will also need the following result from \cite[Proposition 2.6]{gua:ge1} and \cite[Section 2]{hu:ham}.
\begin{proposition}
Let $(\vartheta,X)\in \aut_{CA}(\T M)$, where $X$ is complete with flow $\Phi_s$. Then
the 1-parameter group of automorphisms defined by $(\vartheta,X)$ is
$(\gamma_s,\Phi_s)$, where
\[ \gamma_s=\int_0^s (\Phi_u^{-1})^*\vartheta\ \d u.\]
\end{proposition}
If $E\subset \T M$ is a Dirac structure, and $X+\alpha\in \Gamma(E)$, then
$(\d\alpha+\iota_X\eta,\,X)\in\aut(\T M)$ preserves $E$, hence so does the resulting flow
$(\gamma_s,\Phi_s)$. That is, $\ca{R}_{-\gamma_s}((\T\Phi_s)(E))=E$, or
\begin{equation}\label{eq:gaugetr}
(\T\Phi_s)(E)=E^{\gamma_s}.
\end{equation}

\subsection{Proof of Theorem \ref{th:dirac}}\label{subsec:proofdirac}
\begin{proof}
Given $\eps\in\Gamma(E)$ as in Theorem \ref{th:dirac}, write $\eps=X+\alpha$ where
$X\in\mf{X}(M)$ is the vector field part and $\alpha\in \Om^1(M)$ the 1-form part.
By construction, $X$ is Euler-like along $N$, hence its flow $\Phi_s$  defines a tubular neighborhood embedding $\psi$ such that $\lambda_t\circ \psi=\psi\circ \kappa_t$,
for $t>0$, where we write $\lambda_t=\Phi_{{-\log(t)}}$.
Let
\[ \omega_t=\int_t^1 \f{1}{\tau} \kappa_\tau^*\psi^* (\d\alpha+\iota_X\eta)\,\d\tau,\]
so that $\omega=\omega_0$. We claim that the family of Dirac structures
\begin{equation}\label{eq:fam2}
 (\kappa_t^! \psi^! E)^{\omega_t}\subset \T \nu_N
 \end{equation}
is independent of $t\ge 0$.
This proves the theorem, because \eqref{eq:fam2} equals
 $\psi^! E=(\T \psi^{-1}) (E|_U)$ for $t=1$, and $(\kappa_0^!\psi^! E)^{\omega}
 =(p^!i^! E)^{\omega}$ for $t=0$.

By continuity, it suffices to prove the $t$-independence
 of \eqref{eq:fam2} for $t>0$. By \eqref{eq:gaugetr},  we have
\[ \lambda_t^! E=(\T \lambda_t^{-1})(E)=E^{\gamma_s}\]
for $s=-\log(t)$, where $\gamma_s\in \Om^2(U)$ are the 2-forms
\[ \gamma_s=\int_0^{-\log(t)}(\Phi_{u}^{-1})^*(\d\alpha+\iota_X\eta)\ \d u
=-\int_t^1 \f{1}{\tau} \lambda_\tau^*(\d\alpha+\iota_X\eta)\,\d\tau
.\]
With $\psi^*\gamma_s=-\omega_t$, it follows that
\[ \kappa_t^!\psi^!E=\psi^!\lambda_t^!E=
\psi^! E^{\gamma_s}
=(\psi^!E)^{-\om_t},\]
hence $(\kappa_t^!\psi^!E)^{\om_t}=\psi^! E$ is independent of $t$, as desired. \qedhere
\end{proof}

\subsection{Functorial properties of the normal form}\label{subsec:diracfunc}
Let $\Phi\colon M'\to M$ be a smooth map, and $E\subset \T M$ and $E'\subset \T M'$ Dirac structures relative to closed 3-forms
$\eta$ and $\eta'$. Then $\Phi$ is called a \emph{Dirac morphism} if $\Phi^*\eta=\eta'$, and for all $m=\Phi(m')$ and $v+\mu \in E_m$, there exists a {\em unique} element $v'+\mu'\in E'_{m'}$ such that $v=T\Phi(v')$ and $\mu'=\Phi^*\mu$. We denote such a Dirac morphism by
\[ \T\Phi\colon (\T M',E')\dasharrow (\T M,E).\]
It determines a
comorphism of Lie algebroids $E'\dasharrow E$; the corresponding map $\Phi^*E\to E'$ takes $v+\mu\in E$ to the unique
$v'+\mu'\in E'$ to which it is related.

Given a transversal $N\subset M$ for the Dirac structure $E$, and suppose that the map $\Phi$ is transverse to $N$. Then $N'=\Phi^{-1}(N)$ is a transversal for $E'$,
and given a section $\epsilon \in \Gamma(E)$ such that $X=\a(\epsilon)$ is Euler-like,
then its pull-back $\epsilon'\in\Gamma(E')$ defines an Euler-like vector field $X'=\a(\epsilon')$. Let $\omega$ be as in Theorem \ref{th:dirac}, and $\omega'$ its pull-back under $\nu(\Phi)$. (Equivalently, $\omega'$ is given by Equation \eqref{eq:omega}, using $\alpha'=\Phi^*\alpha$.)
We obtain a commutative diagram of Dirac morphisms,
\begin{equation}\label{diag:funct5} \xymatrixcolsep{50pt}
\xymatrix{ (\T\nu_{N'},\ ({p'}^!{i'}^!E')^{\omega'}) \ar@{-->}[r]_{\T\psi'}\ar@{-->}[d]_{\T \nu({\Phi})}    & (\T M',\ E')\ar@{-->}[d]^{{\T\Phi}}\\{(\T\nu_N,(p^!i^!E)^\omega)}\ar@{-->}[r]_{\T\psi} & (\T M,E).}\end{equation}

\subsection{Uniqueness properties for transversals}\label{subsec:uni3}
Let $E\subset \T M$ be a Dirac structure relative to a closed 3-form $\eta\in \Omega^3(M)$, and $i_q\colon N_q\to M$ a family of transversals labeled by points $q\in Q$. As in Section \ref{subsec:uni1},  $N_q$ are the fibers of a submersion $\psi\colon N\to Q$, and $i_q=\phi\circ j_q$,  where $j_q\colon N_q\to N$ is the inclusion.  Since $\phi$ is transverse to $\a$, it defines an Dirac structure
$F=\phi^!E\subset \T N$, and we have $i_q^!E=j_q^!F$. (Our notation here differs slightly from Section \ref{subsec:uni1}.)

Given any vector field $Z$ on $Q$, we can find a section $Y+\beta\in \Gamma(F)$ such that $Y\sim_\psi Z$. The section defines $(\d\beta+\iota_Y\phi^*\eta,\,Y)\in \aut_{CA}(\T N)$ preserving $F$. To simplify the discussion, let us assume that $Z$ and $Y$ are complete, with flows $\Phi_s^Z,\ \Phi_s^Y$ (in the general case, one has to work with local flows). Then the infinitesimal
automorphism integrates to a 1-parameter group of automorphisms
$(\gamma_s,\Phi_s^Y)$ where
\begin{equation}\label{eq:gammas}
 \gamma_s=\int_0^s (\Phi_{-u}^Y)^*(\d\beta+\iota_Y\phi^*\eta)\ \d u\in \Om^2(N).\end{equation}
As explained above, we have $F=(\Phi_s^Y)^! (F^{\gamma_s})$. Applying
$j_q^!$ to both sides and using that $\Phi_s^Y\circ j_q=j_{\Phi_s^Z(q)}\circ (\Phi_s^Y|_{N_q})$, we obtain
\[
 i_q^!E=j_q^!F=(\Phi_s^Y|_{N_q})^!j_{\Phi_s^Z(q)}^!(F^{\gamma_s})
=(\Phi_s^Y|_{N_q})^!\,(i_{\Phi_s(q)}^! E )^{\vartheta_s},\]
where
\begin{equation}\label{eq:iqe}\vartheta_s=(j_{\Phi_s^Z(q)})^*\gamma_s\in \Omega^2(N_{\Phi_s^Z(q)}).\end{equation}
That is, $\T (\Phi_s^Y|_{N_q})$ gives an isomorphism of Dirac structures
$i_q^!E\to (i_{\Phi_s(q)}^! E )^{\vartheta_s}$.

As a special case, given a leaf $S\subset M$ of the Dirac structure, and two transversals $i_0\colon N_0\to  M$ and $i_1\colon N_1\subset M$, intersecting $S$ in points $m_0,m_1$, we can extend to a family of transversals
$i_s\colon N_s\to M$ with $N_s\cap S=\{m_s\}$, and there is a family of isomorphisms
\begin{equation}\label{eq:famisom} i_0^!E\to (i_s^!E)^{\vartheta_s}.\end{equation}
(Cf. Corollary \ref{cor:uniqlocal}.)

\begin{remark}
Suppose $\eta=0$, and let  $i\colon N\to M$ be a transversal through a given point $m\in M$,
with $T_mM=T_mN\oplus \a(E_m)$. Then the Dirac structure $i^!E$ on $N$ is in fact
a \emph{Poisson structure} near $m$. A uniqueness theorem for these transverse Poisson structure was obtained by Dufour-Wade \cite[Theorem 4.5]{duf:lo}. It can be recovered from
our result, using the argument in Remark \ref{rem:w} \eqref{it:wc} below. (We are grateful to the referee for this remark.)
\end{remark}

\section{Poisson manifolds}\label{sec:poisson}
Let $(M,\pi)$ be a Poisson manifold. We denote by $\pi^\sharp\colon  T^*M\to TM$ the bundle map defined by $\pi$, and by
$E=\on{Gr}(\pi)\subset \T M$ the Dirac structure given as its graph. A submanifold $i\colon N\hra M$ is a
\emph{transversal} for $(M,\pi)$ if it is transverse to the map $\pi^\sharp$. Equivalently, the restriction of  $\pi^\sharp$ to $\on{ann}(TN)$ is injective. The Poisson bivector restricts to a
skew-symmetric bilinear form on the conormal bundle $\on{ann}(TN)\subset T^*M|_N$. The transversal $N$ has \emph{constant corank}
if this restriction has constant rank; these are special cases of the \emph{pre-Poisson submanifolds} studied in the work of Cattaneo-Zambon \cite{cat:coi} and Calvo-Falceto \cite{cal:poi}. If the bilinear form on $\on{ann}(TN)$ is non-degenerate, then $N$ is called a \emph{cosymplectic transversal}; these are discussed in work of Weinstein \cite{wei:loc}, Xu \cite{xu:dir}, Cattaneo-Zambon \cite{cat:coi},
and Frejlich-M\u{a}rcu\c{t} \cite{fre:nor} (under the name of \emph{Poisson transversal}).

For a cosymplectic transversal, the subbundle $\nu_N^*\cong \on{ann}(TN)\subset T^*M|_N$ is a
symplectic vector bundle, with the  fiberwise symplectic structure inverse to the
restriction of $\pi|_N$. The range of
$\pi^\sharp\colon \on{ann}(TN)\to TM|_N$ is a complement to $TN$, identifying $\pi^\sharp(\on{ann}(TN))\cong \nu_N$.
The non-degeneracy condition is equivalent to  the  direct sum decomposition
\begin{equation}\label{eq:split2a}
TM|_N=\pi^\sharp(\on{ann}(TN))\oplus TN,\end{equation}
or dually
\begin{equation}\label{eq:split2}
T^*M|_N=\on{ann}(TN)\oplus T^*N.
\end{equation}
Weinstein \cite[Proposition 1.4]{wei:loc} showed that any cosymplectic transversal $N$ inherits a Poisson structure $\pi_N$.  In fact, we have:
\begin{lemma}
A transversal $i\colon N\hra M$ for a Poisson manifold $(M,\pi)$ is cosymplectic if and only if  the Dirac structure $i^!E\subset \T N$ has trivial  intersection with $TN$.
In this case,  $i^!E=\on{Gr}(\pi_N)$.
\end{lemma}
\begin{proof}
By definition, $i^!E$ consists of all $v'+\mu'\in \T N$ such that
there exists $v+\mu\in \on{Gr}(\pi)$ with $v=Ti(v')$ and
$\mu'=i^*\mu$. Hence, $i^!E\cap TN=0$ holds if and only if the
conditions $i^*\mu=0$ and $\pi^\sharp(\mu)\in TN$ imply
that $\mu=0$.
But this is exactly the condition  $\pi^\sharp(\on{ann}(TN))\cap TN=0$ for a cosymplectic transversal.
\end{proof}
Suppose $N\subset M$ is a cosymplectic transversal.  Choose a 1-form $\alpha\in \Om^1(M)$ with $\alpha|_N=0$, such that the splitting given by the normal derivative $\d^N\alpha\colon \nu_N\to T^*M|_N$ coincides with the given inclusion of $\nu_N\cong \nu_N^*$. Then $X=\pi^\sharp(\alpha)$  has linearization equal to the
Euler-vector field on $\nu_N$. Multiplying $\alpha$ by a suitable bump function, we
 may arrange that $X=\pi^\sharp(\alpha)$ is complete. According to Theorem \ref{th:dirac}, the section $\epsilon=X+\alpha\in \Gamma(\T M)$  gives a local model
\begin{equation}\label{eq:poimodel}
(p^!\on{Gr}(\pi_N))^{\omega}\cong  \on{Gr}(\pi)|_U.
\end{equation}
Here $\psi\colon \nu_N\to U\subset M$ is the tubular neighborhood embedding defined
by the Euler-like vector field $X$, and
\begin{equation}\label{eq:omega}
  \omega=\int_0^1 \f{1}{\tau}\kappa_\tau^* \psi^*(\d\alpha)\ \d\tau
=\psi^* \int_0^1  \f{1}{\tau}\lambda_\tau^* (\d\alpha)\ \d\tau\end{equation}
The presymplectic leaves of the Dirac structure $p^!\on{Gr}(\pi_N)$
are the pre-images under $p$ of the symplectic leaves of $(N,\pi_N)$,
with the 2-forms obtained by pullback under $p$.
The Dirac structure $(p^!\on{Gr}(\pi_N))^\omega$ has the same leaves,
but with the pullback of $\omega$ added to the 2-forms on the leaves.
Let us describe the restriction of this 2-form to $TM|_N=TN\oplus \nu_N$.
\begin{lemma}
 The restriction  of $\omega$ to
 to $T\nu|_N$ has kernel $TN$, and equals the given symplectic form on $\nu_N$.
\end{lemma}
\begin{proof}
Since $\alpha|_N=0$, with normal derivative taking values in $\on{ann}(TN)$, the kernel of $\d\alpha|_N$ contains $TN$. The same is thus true for all $\f{1}{\tau}\lambda_\tau^*\d\alpha$, and hence for the 2-form $\omega$. Due to our choice of $\epsilon$, the differential $T\psi|_N\colon T\nu_N|_N\to TM|_N$
respects the decompositions
\[ TM|_N=TN\oplus \nu_N=T\nu_N|_N\]
where we identify $\pi^\sharp(\on{ann}(TN))=\nu_N$.
Together with the dual decompositions of the cotangent bundles, this means that
$\T\psi$ respects the decompositions
\[ \T M|_N=\T N\oplus (\nu_N\oplus \nu_N^*)=\T\nu_N|_N. \]
The subbundle $\on{Gr}(\pi)|_N\subset \T M|_N$
splits as the direct sum of $\on{Gr}(\pi_N)$ and the graph of the
symplectic form on $\nu_N$. Similarly, $p^!\on{Gr}(\pi_N)|_N\subset
\T\nu_N|_N$ is the direct sum of
$\on{Gr}(\pi_N)$ and $T\nu_N|_N$. Since $\omega|_N$ has kernel $TN$,
the $B$-field transform by $\omega|_N$ preserves this decomposition,
and is trivial on the first summand.
On the other hand, by \eqref{eq:poimodel} it takes
$p^!\on{Gr}(\pi_N)|_N$ to $\psi^!\on{Gr}(\pi)|_N$. This means that
$\omega|_N$ is just the given symplectic structure on $\nu_N\subset
T\nu_N|_N$.
\end{proof}
This allows us to recover the following result.
\begin{theorem}[Frejlich-M\u{a}rcu\c{t} \cite{fre:nor}]
Let $N\subset M$ be a cosymplectic transversal. Choose
 a closed 2-form $\omega\in \Om^2(\nu_N)$ on
the normal bundle, such that $\omega|_N$ has kernel $TN$ and restricts to the given symplectic form on $\nu_N\subset T\nu_N|_N$.
Then, near the zero section of $\nu_N$,
\[ (p^!\on{Gr}(\pi_N))^\omega\]
is the graph of  a Poisson structure, and there exists a tubular neighborhood embedding
$\psi\colon \nu_N\to M$, which is a Poisson map on some neighborhood of $N$.
\end{theorem}
\begin{proof}
We have proved the result for a particular $\omega$  (given by \eqref{eq:omega}). For the general case, note that it suffices to consider closed 2-forms defined
on an open neighborhood
$N\subset \nu_N$. Given two 2-form $\om,\om'$ as in the theorem, one has
$\omega'-\omega=\d\beta$, where $\beta$ is a 1-form, with
$\beta|_N=0$. (The homotopy operator for the retraction from $\nu_N$
to $N$ gives a canonical choice for $\beta$.) The Moser method for
Poisson manifolds (as in, e.g., \cite{al:gw}) gives a Poisson isomorphism
between the models over some neighborhood of the zero section.
\end{proof}
\begin{remark}
The argument from \cite{fre:nor} relies on Crainic-M\u{a}rcu\c{t}'s approach \cite{cra:exi} to symplectic realizations via \emph{Poisson sprays} on $T^*M$.
\end{remark}

\begin{remarks}\label{rem:w}
\begin{enumerate}
\item
The 2-form $\omega$ used in  \eqref{eq:poimodel} gives a Poisson structure over all of $\nu_N$, not only near the zero section. An alternative choice of $\omega$ uses the `minimal coupling'
procedure of Sternberg \cite{ste:min} and Weinstein \cite{wei:uni},
depending on the choice of a symplectic connection on the symplectic vector bundle $\nu_N$.
\item\label{it:wb}
If the normal bundle is \emph{trivial}, $\nu_N=N\times \R^{2k}$, we
may use the trivial connection. The normal form then says that a
neighborhood of $N$ in $M$ is Poisson diffeomorphic to a neighborhood
of $N\times\{0\}$ inside the product of Poisson manifolds,
\[ (N\times \R^{2k},\ \pi_N+\pi_0)\]
where $\pi_0=-\sum_i \f{\p}{\p q_i}\wedge \f{\p}{\p p^i}$ is the standard Poisson structure on $\R^{2k}$.
In particular, given any $m\in M$, consider a submanifold $N$
through $m$
such that
$T_mM=T_mN\oplus \pi^\sharp(T^*_mM)$.
Taking $N$ smaller if necessary, this submanifold is a cosymplectic transversal. One
refers to $\pi_N$ as the \emph{transverse Poisson structure}. We
recover the  \emph{Weinstein splitting theorem}
\cite{wei:loc}, identifying $U\subset M$ with an open
neighborhood of the direct product of Poisson manifolds $N\times
\R^{2k}$.
\item \label{it:wc}
Weinstein's uniqueness result \cite[Lemma~2.2]{wei:loc} for the transverse Poisson structure 
can be recovered from the more general result in Section~\ref{subsec:uni3}, applied to $E=\on{Gr}(\pi)$.
Consider a symplectic leaf $S$ and two transversals $N_0$ and $N_1$, such that $N_0\cap S = \{ m_0\}$ and $N_1\cap S=\{m_1\}$.
As discussed at the end of Section~\ref{subsec:uni3}, $N_0,N_1$ extend to a smooth family of transversals $i_s\colon N_s\to M$,
with $N_s\cap S=\{m_s\}$, and a family of isomorphisms $i_0^!E\cong (i_s^!E)^{\vartheta_s}$, with base maps
$\varphi_s\colon N_0\to N_s$,  for a suitable family of closed 2-forms $\vartheta_s\in\Om^2(N_s)$.
See Equation \eqref{eq:gaugetr}.  By the explicit formula \eqref{eq:gammas}, the forms $\vartheta_s$ are exact,
with a smooth family of primitives $\beta_s$. Equivalently, we obtain a family of isomorphisms
\[ (i_0^!E)^{-d\alpha_s}\cong i_s^!E,\]
where $\alpha_s=\varphi_s^* \beta_s$ are the pullbacks with respect to the underlying diffeomorphisms.
Since $i_s^!E=\on{Gr}(\pi_s)$ is the graph of the induced Poisson structure on $N_s$, this shows that the diffeomorphism $\varphi_s$
is a Poisson map, \emph{up to} a gauge transformation of $\pi_0$ by the exact 2-form $-d\alpha_s$. By the Moser argument for Poisson structures (see e.g. \cite[Section 3.3]{al:gw} or \cite[Section 1.3]{al:lin}), the form $\alpha_s$ defines a time dependent vector field on $N_0$ whose flow $\psi_s$ (defined on a sufficiently small neighborhood of $m_0$) intertwines  the gauge transformed Poisson structures.
Its composition with $\varphi_s$ gives a family of  Poisson diffeomorphisms $(N_0,\pi_0)\to (N_s,\pi_s)$.
\end{enumerate}
\end{remarks}

We also recover the functorial properties of the normal form, as in \cite{fre:nor2}.
Let $N\subset M$ and $\alpha$ as above. Suppose that $(M',\pi')$ is another Poisson manifold, and $\Phi\colon M'\to M$ is a Poisson map transverse to $N$.
Then the pre-image $N'=\Phi^{-1}(N)\subset M$ is a Poisson transversal, and the  pull-back $\alpha'=\Phi^*\alpha$ defines an Euler-like vector field $X'={\pi'}^\sharp\alpha'$, and hence a tubular neighborhood embedding $\psi'$. Since $X'\sim_{\Phi} X$,  we have that $\psi\circ \nu(\Phi)=\Phi\circ \psi'$. Since $\Phi$ is a Poisson map and since $\psi,\psi'$ are Poisson diffeomorphisms onto their images, it is immediate that the map $\nu(\Phi)$
between models  is Poisson. Equivalently, this follows because the 2-forms
are related by $\omega'=\nu(\Phi)^*\omega$.

\section{Generalized complex manifolds}\label{sec:gcs}
Let $\T M$ be equipped with the Courant bracket for the zero 3-form
$\eta=0$. Following
Hitchin \cite{hi:gen} and  Gualtieri \cite{gua:ge, gua:ge1}, one
defines a \emph{generalized complex structure} on $M$ to be a vector
bundle automorphism $\mathbb{J}\in \on{Aut}(\T M)$ with
$\mathbb{J}^2=-\on{id}$, such that $\mathbb{J}$ is orthogonal
(preserves the metric) and such that its $+\sqrt{-1}$ eigenbundle
$E=\ker(\mathbb{J}-\sqrt{-1}\on{id})\subset \T_\C M$ is a Dirac
structure for the complexified Courant bracket and metric. Conversely,
a generalized complex structure may be regarded as a complex Dirac
structure $E\subset \T_\C M$ such that  $E\cap \ol{E}=0$. An ordinary
complex structure $J$ on $TM$ defines a
generalized complex structure of \emph{complex type}, where $\mathbb{J}=J\oplus (J^{-1})^*$. At the opposite
extreme, any symplectic form $\omega$ on $M$ defines a generalized
complex structure of \emph{symplectic type}, with
$E=\on{Gr}(\sqrt{-1}\omega)$,
the graph of the  imaginary 2-form $\sqrt{-1}\omega$. If $\gamma \in\Om^2(M)$ is any
closed real 2-form, and
$E\subset \T_\C M$ is a generalized complex structure, then the
$B$-field transform
$\ca{R}_\gamma(E)$ is again a  generalized complex structure.

Any generalized complex structure $\mathbb{J}$ determines a Poisson
structure $\pi$ on $M$, by
\[ \pi^\sharp(\mu)=\a(\mathbb{J}\mu)\]
for all $\mu\in T^*M\subset \T M$.
See \cite[Section 3.4]{gua:ge1}.
This Poisson structure satisfies
\[ \on{ran}(\pi^\sharp)_\C=\a(E)\cap \a(\ol{E}).\]
If $\mathbb{J}$ is of complex type, then $\pi=0$, while
for $\mathbb{J}$
of symplectic type the Poisson structure is inverse to the given
symplectic form.

\begin{lemma}
Suppose that $i\colon N\hra M$ is a cosymplectic transversal with respect to $\pi$. Then $i^!E$ defines a generalized complex structure on $N$.
\end{lemma}
\begin{proof}
Suppose $v+\mu\in \T M$ lies in the intersection $i^!E\cap \ol{i^!E}$. We want to show that $v=0$ and $\mu=0$.
By treating real and imaginary parts separately, we may assume that
$v=\ol{v}$ and $\mu=\ol{\mu}$.
By definition of $i^!E$, we have that $v\in TS$, and there exists
$\lambda\in T^*_\C M$ with $v+\lambda\in E$ and $i^*\lambda=\mu$.
Taking the imaginary part of
\[ \a(\mathbb{J}(v+\lambda))
=\a(\sqrt{-1}(v+\lambda))=\sqrt{-1} v\]
we see that $\pi^\sharp(\on{Im}(\lambda))=v\in TN$. On the other hand,
taking the imaginary part of $i^*\lambda=\mu$ we get
$i^*\on{Im}(\lambda)=0$, hence $\on{Im}(\lambda)\in \on{ann}(TN)$. By
definition of cosymplectic, this shows that $\on{Im}(\lambda)=0$.
We conclude that $v+\lambda\in E$ is real, and therefore zero. Hence
also $v+\mu=0$, which proves that $ i^!E\cap \ol{i^!E}=0$.
\end{proof}

Letting $p\colon \nu_N\to N$ be the bundle projection as before, the pull-back
$p^!i^!E$ does not define a generalized complex structure, since it
contains the real subbundle
$\ker(Tp)$. However, if $\omega$ is a closed 2-form on $\nu_N$ whose
restriction to
$TM|_N=TN\oplus \nu_N$ has kernel $TN$ and coincides with the  given
form $\omega_0$ on the symplectic vector bundle $\nu_N$, then the
$B$-field transform $(p^!i^!E)^{\sqrt{-1}\omega}$ is a generalized
complex structure on some open neighborhood of $N$. Indeed,
\[p^!i^!E|_N=i^!E\oplus \nu_N\subset \T N\oplus (\nu_N\oplus \nu_N^*),\]
and the gauge transform takes this to $i^!E\oplus \on{Gr}(\sqrt{-1}\omega_0)$.

A version of the following result was independently obtained by Bailey-Cavalcanti-Duran \cite[Section 3.2]{bai:blo}.

\begin{theorem}
Let $E\subset \T_\C M$ be a generalized complex structure, and $N\subset M$ a
cosymplectic transversal for the underlying
Poisson structure  $\pi$. Choose a
1-form $\alpha\in \Gamma(TM)$ as in Section \ref{sec:poisson}, defining an
Euler-like vector field $X=\pi^\sharp(\alpha)$ with corresponding tubular neighborhood embedding $\psi\colon \nu_N\to M$. Then $\psi^*E\subset \T \nu_N$ equals, up
to gauge transformation by a closed real 2-form $\gamma\in \Om^2(\nu_N)$
(defined below), the generalized complex structure
\[ (p^!i^!E)^{\sqrt{-1}\omega} \subset \T_\C\nu_N\]
where $\omega\in \Om^2(\nu_N)$ is the closed 2-form \eqref{eq:omega}.
\end{theorem}
\begin{proof}
By the results
for Poisson manifolds (Section \ref{sec:poisson}), the pre-image of $\on{Gr}(\pi)\subset \T M$ under $\T \psi$ is $(p^!\on{Gr}(\pi_N))^\omega\subset \T\nu_N$.
The Euler-like vector field  $X=\pi^\sharp(\alpha)$ lifts to a section of $E$,
given as \[\epsilon:=(\mathbb{J}+\sqrt{-1}\on{id})\alpha
=X+\beta+\sqrt{-1}\alpha,
\]
where the real 1-form $\beta$ is defined by this equation. (The
definition of $\epsilon$ implies
that $\mathbb{J}\epsilon=\sqrt{-1}\epsilon$, as well as
$\a(\epsilon)=\a(\mathbb{J} \alpha)=\pi^\sharp(\alpha)=X$.) Let
$\gamma\in \Om^2(\nu_N)$ be the real 2-form defined similarly to $\omega$
(see Equation \eqref{eq:omega}), but with $\alpha$ replaced by
$\beta$:
\[ \gamma=\int_0^1 \f{1}{\tau}\kappa_\tau^* \psi^*(\d\beta)\ \d\tau .
\]
 The normal form theorem for Dirac structures (Theorem \ref{th:dirac})
shows that the pre-image of $E$ under the complexified map $T_\C\psi\colon \T_\C\nu_N\to \T_\C M$ is
\[  \psi^!E=(p^!i^!E)^{\gamma+\sqrt{-1}\omega}=((p^!i^!E)^{\sqrt{-1}\omega})^\gamma\]
as subbundles of $\T_\C\nu_N$. Since $E$ is a generalized
complex structure,
$(p^!i^!E)^{\sqrt{-1}\omega}=(\psi^!E)^{-\gamma}$ is one also.
\end{proof}

Suppose that the normal bundle $\nu_N$ is trivial. By the Weinstein
splitting theorem (cf.~ Remark \ref{rem:w} \eqref{it:wb}), one obtains a Poisson isomorphism of a neighborhood
of $N$ in $M$ with a neighborhood of $N\times \{0\}$ inside $N\times
\R^{2k}$, where $N$ has the Poisson structure $\pi_N$, and  $\R^{2k}$
has its standard linear Poisson structure $\pi_0=-\sum_i \f{\p}{\p
q_i}\wedge \f{\p}{\p p^i}$, inverse to the symplectic structure
$\omega_0=\sum_i \d q_i\wedge \d p^i$. Using this model as a starting
point, we may take
$\alpha=\sum_i (q_i \d p^i -p^i \d q_i)$. We then obtain $\d\alpha=2\om_0$,
hence $\f{1}{\tau}\kappa_\tau^*\d\alpha=2\tau \om_0$,  and finally
$\omega=\omega_0$.
In particular, we recover the splitting theorem  for generalized
complex manifolds, due to  Abouzaid-Boyarchenko \cite[Theorem
1.4]{abo:lo}:
\begin{corollary}
Let $M$ be a generalized complex manifold, with underlying Poisson
structure $\pi$,
and $m\in M$. Put  $\P=\on{ran}(\pi_m^\sharp)$, and let
$N\subset M$ be a submanifold containing $m$, and such that
$T_mM=T_mN\oplus \P$. Give
$\P$ the generalized complex structure corresponding to its symplectic
form, and give
$N$ the generalized complex structure with corresponding Dirac
structure $i^!E$.
Up to a $B$-field transform, there is an isomorphism of generalized
complex manifolds from a neighborhood of $m$ in $M$ and in the product
$N\times \P$.
\end{corollary}

\begin{appendix}

\section{Normal bundles of vector subbundles}\label{app:tnu}
For a vector bundle $\pr: E\to M$, with a vector subbundle $F\to N$
along a submanifold $N\subseteq M$, the normal bundle $\nu(E,F)$ is a vector bundle
over $\nu(M,N)$ with projection $\nu(\pr): \nu(E,F)\to \nu(M,N)$.
In fact, $\nu(E,F)$ fits into a {\em double vector bundle}
\[
\xymatrix{
{\nu(E,F)} \ar[r]^{}\ar[d]_{\nu(\pr)}    & F \ar[d]\\
{\nu(M,N)}\ar[r]_{} & N, }
\]
that is, in this diagram both horizontal and vertical arrows are vector-bundle
projections, and the horizontal and vertical scalar multiplications
commute (see \cite{gra:hig} for this characterization of double vector bundles).
In particular, for a submanifold $N\subseteq M$, we have a double vector bundle
\[
\xymatrix{
{\nu(TM,TN)} \ar[r]^{}\ar[d]_{}    & TN \ar[d]\\
{\nu(M,N)}\ar[r]_{} & N. }
\]
The tangent bundle to $\nu(M,N)$ also gives rise to a double vector bundle, the so-called {\em tangent prolongation} of $p:\nu(M,N)\to N$:
\[
\xymatrix{
{T\nu(M,N)} \ar[r]^{\;Tp}\ar[d]_{}    & TN \ar[d]\\
{\nu(M,N)}\ar[r]_{} & N. }
\]
\begin{lemma}\label{lem:TnuT}
There is a natural map $\nu(TM,TN)\to T\nu(M,N)$ which is a vector-bundle isomorphism with respect to the vector-bundle structures over $TN$ and $\nu(M,N)$, covering the identity map in each case.
In particular, $\nu(TM,TN)$ and $T\nu(M,N)$ are identified as double vector bundles.
\end{lemma}

\begin{proof}
Let $\pr_M: TM\to M$ denote the tangent bundle to $M$. The iterated tangent bundle $T(TM)$ is a double vector bundle
\[  \xymatrixcolsep{30pt}
\xymatrixrowsep{24pt}
\xymatrix{ TTM\ar[r]^{T\pr_M}\ar[d]_{\pr_{TM}} & TM\ar[d]\\ TM\ar[r] & M.
}
\]
There is a canonical involution $J: TTM\to TTM$ satisfying $(T\pr_M) \circ J = \pr_{TM}$ and which interchanges the vertical and horizontal vector bundle structures. See e.g. \cite[Section 9.6]{mac:gen}.

For a submanifold $N\subseteq M$, the submanifolds $T(TM|_N)$ and $(TTM)|_{TN}$ of $TTM$ are both sub-double vector bundles:
\[\xymatrixcolsep{38pt}
\xymatrixrowsep{24pt}
\xymatrix{T(TM|_N)\ar[r]^{T\pr_M}\ar[d]_{\pr_{TM}} & TN\ar[d]\\ TM|_N\ar[r] & N,
}\  \ \ \ \ \ \ \ \
\xymatrix{(TTM)|_{TN}\ar[r]^{T\pr_M}\ar[d]_{\pr_{TM}} & TM|_N\ar[d]\\ TN\ar[r] & N.
}
\]
The involution $J: TTM\to TTM$ restricts to an isomorphism $T(TM|_N)\to (TTM)|_{TN}$ between these two double vector bundles. This map also restricts to the canonical involution of $TTN$, viewed as submanifolds of
 $T(TM|_N)$ and $(TTM)|_{TN}$. In this way, $J$ gives rise to an isomorphism between the two double vector bundles
 \[
 \xymatrixcolsep{38pt}
\xymatrixrowsep{24pt}
\xymatrix{T(\nu(M,N))\ar[r]^{T p}\ar[d]_{\pr_{\nu(M,N)}} & TN\ar[d]\\ \nu(M,N)\ar[r] & N,
}\  \ \ \ \ \ \ \ \
\xymatrix{\nu(TM,TN)\ar[r]^{p}\ar[d]_{\nu(p)} & \nu(M,N) \ar[d]\\ TN\ar[r] & N,
}
 \]
as desired.

\end{proof}

\end{appendix}

\bibliographystyle{amsplain}


\begin{thebibliography}{10}
\bibitem{abo:lo}
M.~Abouzaid and M.~Boyarchenko, \emph{Local structure of generalized complex
  manifolds}, J. Symplectic Geom. \textbf{4} (2006), no.~1, 43--62.

\bibitem{ab:ma}
R.~Abraham, J.~Marsden, and T.~Ratiu, \emph{Manifolds, tensor analysis and
  applications}, Addison-Wesley, Reading, 1983.

\bibitem{al:lin}
A.~Alekseev and E.~Meinrenken, \emph{{Linearization of Poisson Lie group
  structures}}, J. Symplectic Geom. \textbf{14} (2016), no.~1, 227--267.

\bibitem{al:gw}
\bysame, \emph{Ginzburg-{W}einstein via {G}elfand-{Z}eitlin}, J. Differential
  Geom. \textbf{76} (2007), no.~1, 1--34.

\bibitem{and:hol}
I.~Androulidakis and G.~Skandalis, \emph{The holonomy groupoid of a singular
  foliation}, J. Reine Angew. Math. \textbf{626} (2009), 1--37.

\bibitem{bai:blo}
M.~Bailey, G.~Cavalcanti, and J.~van~der Leer~Duran, \emph{{Blow-ups in
  generalized complex geometry}}, Preprint, 2016, arXiv:1602.02076.

\bibitem{bal:not}
R.~Balan, \emph{A note about integrability of distributions with
  singularities}, Boll. Un. Mat. Ital. A (7) \textbf{8} (1994), no.~3,
  335--344.

\bibitem{bas:lin}
J.~Basto-Goncalves, \emph{Linearization of resonant vector fields}, Trans.
  Amer. Math. Soc. \textbf{362} (2010), no.~12, 6457--6476.


\bibitem{blo:rem}
C.~Blohmann, \emph{Removable presymplectic singularities and the local
  splitting of {D}irac structures}, 2014, arXiv:1410.5298. To appear in Int. Math. Res. Notices.

\bibitem{bur:vec}
H.~Bursztyn,  A.~Cabrera and  M.~ del Hoyo, \emph{{Vector bundles over Lie groupoids and algebroids}}, Adv. in Math. \textbf{290} (2016), 163--207.

\bibitem{cal:poi}
I.~Calvo and F.~Falceto, \emph{{Poisson reduction and branes in Poisson sigma
  models}}, Lett.~Math.~Phys. \textbf{70} (2004), 231---247.

\bibitem{cat:coi}
A.~S. Cattaneo and M.~Zambon, \emph{Coisotropic embeddings in {P}oisson
  manifolds}, Trans. Amer. Math. Soc. \textbf{361} (2009), no.~7, 3721--3746.

\bibitem{cou:di}
T.~Courant, \emph{Dirac manifolds}, Trans.~Amer.~Math.~Soc. \textbf{319}
  (1990), no.~2, 631--661.

\bibitem{cra:exi}
M.~Crainic and I.~Marcut, \emph{On the existence of symplectic realizations},
  J. Symplectic Geom. \textbf{9} (2011), no.~4, 435--444.

\bibitem{dra:smo}
L.~Drager, J.~Lee, E.~Park, and K.~Richardson, \emph{Smooth distributions are
  finitely generated}, Annals of Global Analysis and Geometry \textbf{41}
  (2012), no.~3, 357--369 (English).

\bibitem{duf:nor}
J.-P. Dufour, \emph{Normal forms for {L}ie algebroids}, Lie algebroids and
  related topics in differential geometry ({W}arsaw, 2000), Banach Center
  Publ., vol.~54, Polish Acad. Sci. Inst. Math., Warsaw, 2001, pp.~35--41.

\bibitem{duf:lo}
J.-P. Dufour and A.~Wade, \emph{On the local structure of {D}irac manifolds},
  Compos. Math. \textbf{144} (2008), no.~3, 774--786.

\bibitem{duf:po}
J.-P. Dufour and N.T. Zung, \emph{Poisson structures and their normal forms},
  Progress in Mathematics, vol. 242, Birkh\"auser Verlag, Basel, 2005.

\bibitem{fer:lie}
R.~Fernandes, \emph{Lie algebroids, holonomy and characteristic classes}, Adv.
  Math. \textbf{170} (2002), no.~1, 119--179.

\bibitem{fre:nor}
P.~Frejlich and I.~Marcut, \emph{The local normal form around {P}oisson
  transversals}, 2013, arXiv1306.6055. To appear in Pacific J. Math.

\bibitem{fre:nor2}
\bysame, \emph{{Normal forms for Poisson maps and symplectic groupoids around
  Poisson transversals}}, 2015, arXiv:1508.05670.

\bibitem{gra:hig}
J.~Grabowski and M.~Rotkiewicz, \emph{Higher vector bundles and multi-graded
  symplectic manifolds}, J. Geom. Phys. \textbf{59} (2009), no.~9, 1285--1305.

\bibitem{gua:ge}
M.~Gualtieri, \emph{{Generalized complex geometry}}, Ph.D. thesis, Oxford,
  2004, arXiv:math.DG/0401221.

\bibitem{gua:ge1}
\bysame, \emph{Generalized complex geometry}, Ann. of Math. (2) \textbf{174}
  (2011), no.~1, 75--123.

\bibitem{gu:gea}
V.~Guillemin and S.~Sternberg, \emph{Geometric asymptotics}, revised ed.,
  Mathematical Surveys and Monographs, vol.~14, Amer.~Math.~Soc., Providence,
  R.~I., 1990.

\bibitem{hi:gen}
N.~Hitchin, \emph{Generalized {C}alabi-{Y}au manifolds}, Q.~J.~Math.
  \textbf{54} (2003), no.~3, 281--308.

\bibitem{hu:ham}
S.~Hu, \emph{Hamiltonian symmetries and reduction in generalized geometry},
  Houston J. Math \textbf{35} (2009), no.~3, 787--811.

\bibitem{lan:dif}
S.~Lang, \emph{{Differential and Riemannian Manifolds}}, vol. 160,
  Springer-Verlag, 1995.

\bibitem{lib:cou}
D.~Li-Bland and E.~Meinrenken, \emph{{C}ourant algebroids and {P}oisson
  geometry}, International Mathematics Research Notices \textbf{11} (2009),
  2106--2145.

\bibitem{mac:gen}
K.~Mackenzie, \emph{General theory of {L}ie groupoids and {L}ie algebroids},
  London Mathematical Society Lecture Note Series, vol. 213, Cambridge
  University Press, Cambridge, 2005.

\bibitem{mir:no}
E.~Miranda and N.~T. Zung, \emph{A note on equivariant normal forms of
  {P}oisson structures}, Math. Res. Lett. \textbf{13} (2006), no.~5-6,
  1001--1012.

\bibitem{roy:co}
D.~Roytenberg, \emph{{Courant algebroids, derived brackets and even symplectic
  supermanifolds}}, Thesis, Berkeley 1999. arXiv:math.DG/9910078.

\bibitem{sev:let}
P.~{\v{S}}evera, \emph{{Letters to Alan Weinstein}},
  http://sophia.dtp.fmph.uniba.sk/~severa/letters/, 1998-2000.

\bibitem{sev:poi}
\bysame, \emph{{Poisson Lie T-duality and Courant algebroids}}, Lett.~ Math.~
  Phys. \textbf{105} (2015), no.~12, 1689--1701.

\bibitem{ste:int}
P.~Stefan, \emph{Integrability of systems of vector fields}, J. London Math.
  Soc. (2) \textbf{21} (1980), no.~3, 544--556.

\bibitem{ste:loc}
S.~Sternberg, \emph{Local contractions and a theorem of {P}oincar\'e}, Amer. J.
  Math. \textbf{79} (1957), 809--824.

\bibitem{ste:str}
\bysame, \emph{On the structure of local homeomorphisms of {E}uclidean
  {$n$}-space. {II}.}, Amer. J. Math. \textbf{80} (1958), 623--631.

\bibitem{ste:min}
\bysame, \emph{On minimal coupling and the symplectic mechanics of a classical
  particle in the presence of a {Y}ang-{M}ills field}, Proc.~ Nat.~ Acad.~
  Sci.~ USA \textbf{74} (1977), 5253--5254.

\bibitem{sus:orb}
H.~Sussmann, \emph{Orbits of families of vector fields and integrability of
  distributions}, Trans. Amer. Math. Soc. \textbf{180} (1973), 171--188.

\bibitem{wei:uni}
A.~Weinstein, \emph{A universal phase space for particles in {Y}ang-{M}ills
  fields}, Lett. Math. Phys. \textbf{2} (1978), 417?--420.

\bibitem{wei:loc}
\bysame, \emph{The local structure of {P}oisson manifolds}, J. Differential
  Geom. \textbf{18} (1983), no.~3, 523--557.

\bibitem{wei:alm}
\bysame, \emph{Almost invariant submanifolds for compact group actions}, J.
  Eur. Math. Soc. (JEMS) \textbf{2} (2000), no.~1, 53--86.


\bibitem{xu:dir}
P.~Xu, \emph{Dirac submanifolds and {P}oisson involutions}, Ann. Sci. \'Ecole
  Norm. Sup. (4) \textbf{36} (2003), no.~3, 403--430.

\end{thebibliography}

\def\cprime{$'$} \def\polhk#1{\setbox0=\hbox{#1}{\ooalign{\hidewidth
  \lower1.5ex\hbox{`}\hidewidth\crcr\unhbox0}}} \def\cprime{$'$}
  \def\cprime{$'$} \def\cprime{$'$} \def\cprime{$'$}
  \def\polhk#1{\setbox0=\hbox{#1}{\ooalign{\hidewidth
  \lower1.5ex\hbox{`}\hidewidth\crcr\unhbox0}}} \def\cprime{$'$}
  \def\cprime{$'$} \def\cprime{$'$} \def\cprime{$'$} \def\cprime{$'$}
\providecommand{\bysame}{\leavevmode\hbox to3em{\hrulefill}\thinspace}
\providecommand{\MR}{\relax\ifhmode\unskip\space\fi MR }
\providecommand{\MRhref}[2]{%
  \href{http://www.ams.org/mathscinet-getitem?mr=#1}{#2}
}
\providecommand{\href}[2]{#2}

\end{document}